\documentclass{article}
\usepackage[margin=1.25in]{geometry}
\usepackage{mathtools,amssymb,amsthm}
\usepackage{graphicx,color}
\usepackage{subcaption}
\usepackage{pifont}% http://ctan.org/pkg/pifont
\usepackage{booktabs}
\usepackage[round,sort]{natbib}
\usepackage{float}
\usepackage[boxed]{algorithm2e}
\usepackage[hyphens]{url}
\usepackage{bm}

\usepackage[toc,page]{appendix}
\usepackage{hyperref}
\usepackage[capitalize]{cleveref}
\newtheorem{theorem}{Theorem}
\newtheorem{corollary}{Corollary}
\newtheorem{proposition}{Proposition}
\newtheorem{lemma}{Lemma}
\newtheorem{remark}{Remark}
\theoremstyle{definition}

\usepackage[dvipsnames]{xcolor}
\usepackage{hyperref}
\hypersetup{citecolor=purple, colorlinks=true, linkcolor=blue}

\newcommand{\cC}{\mathcal{C}}

\newcommand{\cE}{\mathcal{E}}
\newcommand{\cF}{\mathcal{F}}
\newcommand{\cG}{\mathcal{G}}
\newcommand{\cH}{\mathcal{H}}

\newcommand{\cJ}{\mathcal{J}}

\newcommand{\cM}{\mathcal{M}}
\newcommand{\cN}{\mathcal{N}}

\newcommand{\cP}{\mathcal{P}}

\newcommand{\cS}{\mathcal{S}}

\newcommand{\DD}{\mathbb{D}}
\newcommand{\EE}{\mathbb{E}}

\newcommand{\NN}{\mathbb{N}}

\newcommand{\PP}{\mathbb{P}}

\newcommand{\RR}{\mathbb{R}}

\newcommand*{\kl}[3][]{%
\ifthenelse{\isempty{#1}}{\operatorname{D}(#2\,\|\,#3)}%
{\operatorname{D}(#2\,\|\,#3\mid#1)}%
}
\newcommand*{\tv}[2]{\mathrm{d_{TV}}(#1, #2)}

\newcommand*{\norm}[1]{\left\|#1\right\|}
\newcommand*{\triplenorm}[1]{{\left\vert\kern-0.25ex\left\vert\kern-0.25ex\left\vert #1
    \right\vert\kern-0.25ex\right\vert\kern-0.25ex\right\vert}}

\newcommand*{\E}{\mathbb E}

\newcommand*{\ep}{\varepsilon}
\newcommand*{\defeq}{\coloneqq}
\newcommand*{\rd}{\mathrm{d}}
\newcommand*{\dd}{\, \rd}

\DeclarePairedDelimiter\floor{\lfloor}{\rfloor}
\newcommand*{\feps}{f_{\ep}}
\newcommand*{\geps}{g_{\ep}}

\newcommand*{\tse}{T_{\eps, (n, n)}}

\newcommand{\R}{\mathbb{R}}
\newcommand{\eps}{\varepsilon}
\newcommand{\pran}[1]{\left(#1\right)}
\newcommand{\brac}[1]{\left[#1\right]}
\newcommand{\sse}{\subseteq}

\DeclareMathOperator{\Ent}{Ent}
\newcommand{\KL}[2]{D_\text{KL}(#1 \| #2)}

\DeclareMathOperator{\supp}{\mathrm{supp}}
\usepackage{titling}

\begin{document}

\begin{center} {\LARGE{{Entropic estimation of optimal transport maps }}}

{\large{

\vspace*{.3in}

\begin{tabular}{cccc}
Aram-Alexandre Pooladian$^*$, Jonathan Niles-Weed$^{*\dagger}$\\
\end{tabular}
 
{
\vspace*{.1in}
\begin{tabular}{c}
				$^*$Center for Data Science, New York University\\
				$^\dagger$Courant Institute of Mathematical Sciences, New York University \\
 				\texttt{ap6599@nyu.edu, jnw@cims.nyu.edu}
\end{tabular} 
}

}}
\vspace*{.1in}

\today

\end{center}

\vspace*{.1in}
\begin{abstract}
We develop a computationally tractable method for estimating the optimal transport map between two distributions over $\R^d$ with rigorous finite-sample guarantees.
Leveraging an entropic version of Brenier's theorem, we show that our estimator---the \emph{barycentric projection} of the optimal entropic plan---is easy to compute using Sinkhorn's algorithm.
As a result, unlike current approaches for map estimation, which are slow to evaluate when the dimension or number of samples is large, our approach is parallelizable and extremely efficient even for massive data sets.
Under smoothness assumptions on the optimal map, we show that our estimator enjoys comparable statistical performance to other estimators in the literature, but with much lower computational cost.
We showcase the efficacy of our proposed estimator through numerical examples, even ones not explicitly covered by our assumptions. By virtue of Lepski's method, we propose a modified version of our estimator that is adaptive to the smoothness of the underlying optimal transport map.
Our proofs are based on a modified duality principle for entropic optimal transport and on a method for approximating optimal entropic plans due to Pal (2019).
\end{abstract}

\section{Introduction}

The goal of optimal transport is to find a map between two probability distributions that minimizes the squared Euclidean transportation cost.
This formulation leads to what is known as the \textit{Monge problem}~\citep{Mon81}:
\begin{align}\label{eq: monge_p}
\min_{T \in \mathcal{T}(P,Q)} \int \frac{1}{2}\| x - T(x)\|^2_2 \dd P(x),
\end{align}
where $P$ and $Q$ are two probability measures on $\Omega \sse \R^d$, and $\mathcal{T}(P,Q)$ is the set of \textit{admissible maps}\footnote{$\mathcal{T}(P,Q) := \{T : \Omega \to \Omega \ | \ T_\sharp P := P \circ T^{-1} = Q\}$}.
A solution to the Monge problem is guaranteed to exist if $P$ and $Q$ have finite second moments and $P$ is absolutely continuous; moreover, the optimal map enjoys certain regularity properties under stricter assumptions on $P$ and $Q$ (see \cref{sec: background_ot} for more information).
Due to their versatility and mathematical simplicity, optimal transport maps have found a wide range of uses in statistics and machine learning, \citep{carlier2016vector,chernozhukov2017monge,ArjChiBot17, huang2021convex, finlay2020learning,CouFlaTui14,CouFlaTui17, wang2010optimal, onken2021ot, makkuva2020optimal}, computer graphics \citep{SolPeyKim16,SolGoePey15,feydy2017optimal}, and computational biology \citep{schiebinger2019optimal,dai2018autoencoder}, among other fields.

In many applications, we are not given access to the full probability measures $P$ and $Q$, but independent samples from them, denoted $X_1, \ldots, X_n \sim P$ and $Y_1, \ldots, Y_n \sim Q$. 
When an optimal map $T_0 \in \mathcal{T}(P,Q)$ minimizing \cref{eq: monge_p} exists, it is natural to ask whether it is possible to estimate $T_0$ on the basis of these samples.
\cite{hutter2021minimax} investigated this question and proposed an estimator $\hat{T}_n$ which achieves 
\begin{align}\label{hr_bound}
\E \|\hat{T}_n - T_0\|^2_{L^2(P)} \lesssim n^{-\frac{2\alpha}{2\alpha - 2 + d}}\log^3(n) \,,
\end{align}
if $T_0 \in \cC^\alpha$, $P$ and $Q$ are compactly supported, and satisfy additional technical assumptions.
Moreover, they showed that the rate in~\cref{hr_bound} is minimax optimal up to logarithmic factors.
Though statistically optimal, their estimator is impractical to compute if $d > 3$, since it relies on a gridding scheme whose computational cost scales exponentially in the dimension.
Recently, \cite{deb2021rates} and \cite{manole2021plugin} proposed plugin estimators that also achieve the minimax estimation rate.
Though simpler to compute than the estimator of~\cite{hutter2021minimax}, these estimators require at least $O(n^3)$ time to compute and cannot easily be parallelized, making them an unfavorable choice when the number of samples is large.

We adopt a different approach by leveraging recent advances in computational optimal transport based on entropic regularization~\citep{PeyCut19}, which replaces~\cref{eq: monge_p} by
\begin{equation}\label{eq:ent_ot}
\inf_{\pi \in \Pi(P, Q)} \iint \frac 12 \|x - y\|^2 \dd \pi(x, y) + \ep \KL{\pi}{P \otimes Q}\,,
\end{equation}
where $\Pi(P, Q)$ denotes the set of couplings between $P$ and $Q$ and $\KL{\cdot}{\cdot}$ denotes the Kullback--Leibler divergence.
This approach, which was popularized by \cite{Cut13}, has been instrumental in the adoption of optimal transport methods in the machine learning community because it leads to a problem that can be solved by Sinkhorn's algorithm~\citep{Sin67}, whose time complexity scales \emph{quadratically} in the number of samples~\citep{AltWeeRig17}.
Moreover, Sinkhorn's algorithm is amenable to parallel implementation on GPUs, making it very attractive for large-scale problems~\citep{feydy2019interpolating,feydy2020fast,GenCutPey16,GenPeyCut18,altschuler2018massively}.

The efficiency and popularity of Sinkhorn's algorithm raise the tantalizing question of whether it is possible to use this practical technique to develop estimators of optimal transport maps with convergence guarantees.
In this work, we develop such a procedure.
Under suitable technical assumptions on $P$ and $Q$, % less strenuous than those imposed by~\cite{hutter2021minimax},
we show that our estimator $\hat T$ enjoys the rate
\begin{equation*}
\EE \|\hat T - T_0\|_{L_2(P)}^2 \lesssim n^{- \frac{(\alpha + 1)}{2(d + \alpha + 1 )}} \log n
\end{equation*}
if the inverse map $T_0^{-1}$ is $\cC^\alpha$ and $\alpha \in (1, 3]$.
This rate is worse than~\cref{hr_bound}, but our empirical results show that our estimator nevertheless outperforms all other estimators proposed in the literature in terms of \emph{both} computational and statistical performance.
The estimator we analyze was originally suggested by~\cite{seguy2017large}, who also showed consistency of the entropic plan in the large-$n$ limit if the regularization parameter is taken to zero sufficiently fast.
However, to our knowledge, our work offers the first finite-sample convergence guarantees for this proposal.

Our estimator is defined as the barycentric projection~\citep{AmbGigSav08} of the entropic optimal coupling between the empirical measures arising from the samples.
The barycentric projection has been leveraged in other works on map estimation as a straightforward way of obtaining a function from a coupling between two probability measures~\citep{deb2021rates}.
However, in the context of entropic optimal transport, this operation has a more canonical interpretation in light of Brenier's theorem~\citep{Bre91}.
Brenier's result says that the solution to~\cref{eq: monge_p} can be realized as the gradient of the function which solves the dual problem to~\cref{eq: monge_p}; we show in \cref{prop: ent_brenier} that the barycentric projection of the entropic optimal coupling is the gradient of the function which solves the dual problem to~\cref{eq:ent_ot}.
In addition to providing a connection to the classical theory of optimal transport, this observation provides a canonical extension $\hat T$ to out-of-sample points.
Moreover, since Sinkhorn's algorithm computes solutions to the dual of~\cref{eq:ent_ot}, this interpretation shows that computing $\hat T$ is no more costly than solving~\cref{eq:ent_ot}. Moreover, we propose a variant of our estimator that is \emph{adaptive} in the sense that the smoothness parameter need not be explicitly known to the practitioner.

We analyze $\hat T$ by employing a strategy pioneered by~\cite{pal2019difference} for understanding the structure of the optimal entropic coupling.
This technique compares the solution to~\cref{eq:ent_ot} to a coupling whose conditional laws are Gaussian, with mean and covariance characterized by the solution to~\cref{eq: monge_p}.
To leverage this comparison, we employ a duality principle in conjuction with an upper bound reminiscent of the short-time expansions of the value of~\cref{eq:ent_ot} developed in~\cite{conforti2021formula} and~\cite{chizat2020faster}.

\subsubsection*{Paper outline}
The paper is organized as follows: \cref{sec: background} reviews the relevant background on optimal transport theory and entropic regularization for the quadratic cost. We define our estimator and preview our main results in \cref{sec: motivation}. 
Our main statistical bounds appear in \cref{sec: one_samp} and \cref{sec: two_samp}. \Cref{sec: adaptive_estimation} contains a version of our estimator that is adaptive to smoothness.
A discussion of computational considerations and numerical experiments are provided in \cref{sec: exp}.  
\subsubsection*{Notation}
The support of a probability distribution is given by $\supp(\cdot)$. For a convex function $\varphi$, we denote its convex dual by $\varphi^*(y) = \sup_x \{x^\top y - \varphi(x)\}.$ The Kullback--Liebler divergence between two measures is denoted by $\KL \mu \nu = \int \log( \frac {\dd \mu} {\dd \nu} )\dd \mu.$ If $P$ possesses a density $p$ with respect to the Lebesgue measure, we denote its differential entropy by $\text{Ent}(P) = \int p(x) \log(p(x)) \dd x$. For a joint probability density $p(x,y)$, we denote the conditional density of $y$ given $x$ as $p^x(y)$. The square-root of the determinant of a matrix is $J(\cdot) := \sqrt{\det(\cdot)}$.
For $\alpha \geq 0$ and a closed set $\Omega$, we write $h \in \cC^{\alpha}(\Omega)$ if there exists an open set $U \supseteq \Omega$ and a function $g: U \to \RR$ such that $g|_{\Omega} = h$ and such that $g$ possesses $\floor{\alpha}$ continuous derivatives and whose $\floor{\alpha}$th derivative is $(\alpha - \floor{\alpha})$-H\"older smooth.
The total variation distance between two probability measures $\mu$ and $\nu$ is $\tv{\mu}{\nu} =\sup_{f : \|f\|_\infty \leq 1}\int f(\rd \mu - \rd \nu)$.
For $a \in \R^d$ and $r > 0$, we write $B_r(a)$ for the Euclidean ball of radius $r$ centered at $a$.
A constant is a quantity whose value may depend on the smoothness parameters appearing in assumptions \textbf{(A1)} to \textbf{(A3)}, the set $\Omega$, and the dimension, but on no other quantities.
We denote the maximum and minimum of $a$ and $b$ by $a \vee b$ and $a \wedge b$, respectively.
We use the symbols $c$ and $C$ to denote positive constants whose value may change from line to line, and write $a\lesssim b$ and $a \asymp b$ if there exists constants $c, C > 0$ such that $a \leq Cb$ and $cb \leq a \leq Cb$, respectively.

Our proofs based on empirical process theory will consider suprema over uncountable collections of random variables; however, since all the processes in question are separable, these suprema are still measurable~\citep[Section 2.1]{GinNic16}.
\section{Background on optimal transport theory}\label{sec: background}
\subsection{Optimal transport under the quadratic cost}\label{sec: background_ot}
Let $\cP(\Omega)$ be the space of (Borel) probability measures with support contained in $\Omega$, and~$\cP_{ac}(\Omega)$ be those with densities with respect to Lebesgue measure. We first present Brenier's Theorem, which guarantees the existence of an optimal map between two distributions when the first measure is absolutely continuous.
\begin{theorem}[Brenier's Theorem]\label{thm: brenier_thm}\citep{Bre91}
Let $P \in \cP_{ac}(\Omega)$ and $Q \in \cP(\Omega)$. Then
\begin{enumerate}
\item there exists a solution $T_0$ to~\cref{eq: monge_p}, with $T_0 = \nabla \varphi_0$, for a convex function $\varphi_0$ solving
\begin{align}\label{eq: semidual}
\inf_{\varphi \in L^1(P)} \int \varphi \dd P + \int \varphi^* \dd Q,
\end{align} 
where $\varphi^*$ is the convex conjugate to $\varphi$.
\item If in addition $Q \in \cP_{ac}(\Omega)$, then $\nabla \varphi^*_0$ is the optimal transport map from $Q$ to $P$.
\end{enumerate}
\end{theorem}
If $P$ does not have a density, then~\cref{eq: monge_p} may not have a solution, but this problem can be remedied by passing to a convex relaxation of \cref{eq: monge_p} due to \cite{Kan42}, which leads to the definition of the 2-Wasserstein distance:
\begin{align}\label{eq: kant_p}
\frac{1}{2}W_2^2(P,Q) := \min_{\pi \in \Pi(P,Q)} \int \frac{1}{2}\| x - y\|_2^2 \dd \pi(x,y),
\end{align}
where 
$$ \Pi(P,Q) := \{ \pi \in \cP(\Omega \times \Omega) \ |  \ \pi(A\times\Omega) = P(A), \pi(\Omega\times A) = Q(A) \}$$
is the set of \textit{couplings} between $P$ and $Q$. Unlike \cref{eq: monge_p}, \cref{eq: kant_p} always admits a minimizer when $P$ and $Q$ have finite second moments. We call such a minimizer an \textit{optimal plan}, denoted $\pi_0$.

\Cref{eq: kant_p} also possesses a dual formulation \citep[see][]{Vil08,San15}:
\begin{align}\label{eq: ot_dual}
\frac{1}{2}W^2_2(P,Q) = \sup_{(f,g)\in\Phi}\int f \dd P + \int g \dd Q,
\end{align}
where 
$$ \Phi  = \left\{(f,g) \in L^1(P)\times L^1(Q) \ | \ f(x) + g(y) \leq \frac{1}{2}\|x - y\|_2^2 \ \text{for all } x,y \in \Omega\right\}. $$
Once again, if $P$ and $Q$ have finite second moments, then the supremum in \cref{eq: ot_dual} is always achieved by a pair $(f_0,g_0) \in \Phi$, called \textit{optimal potentials}.
\begin{remark}
	It is straightforward to see that \cref{eq: ot_dual} and \cref{eq: semidual} are in fact explicitly linked: if $(f_0,g_0)$ are solutions to \cref{eq: ot_dual}, then  $\varphi_0 = \frac{1}{2}\|\cdot\|_2^2 - f_0$ solves \cref{eq: semidual} and, if $P \in \cP_{ac}$, $T_0 = \text{Id} - \nabla f_0$ solves~\cref{eq: monge_p}.
\end{remark}

\subsection{Entropic optimal transport under the quadratic cost}\label{sec: background_eot}
\textit{Entropic optimal transport} is defined by adding an entropic regularization term to \cref{eq: kant_p} \citep{Cut13}. Letting $P,Q \in \cP(\Omega)$ and a regularization parameter $\eps > 0$, the entropically regularized problem is
\begin{align}\label{eq: primal_eot}
S_\eps(P,Q) &:= \inf_{\pi \in \Pi(P,Q)} \iint \frac{1}{2}\|x-y\|_2^2 \dd \pi(x,y) + \eps \KL{\pi}{P \otimes Q}\,, \\
\intertext{which, when $P$ and $Q$ have densities $p$ and $q$ respectively, can also be written }
S_\eps(P,Q) &= \inf_{\pi \in \Pi(P,Q)} \iint \frac{1}{2}\|x-y\|_2^2 \dd \pi(x,y) + \eps \iint \log(\pi)\dd\pi - \eps \iint \log(p(x)q(y)) \dd \pi(x,y) \nonumber \\
&= \inf_{\pi \in \Pi(P,Q)} \iint \frac{1}{2}\|x-y\|_2^2 \dd\pi(x,y) + \eps \iint \log(\pi)\dd\pi - \eps(\Ent{P} + \Ent{Q})\,.\nonumber
\end{align}
This problem also admits a dual~\citep[see][]{genevay2019entropy}, which is a relaxed version of \cref{eq: ot_dual}:
\begin{align}\label{eq: dual_eot}
S_\ep(P, Q) = \sup_{\substack{f \in L^1(P)\\g \in L^1(Q)}} \int f \dd P + \int g \dd Q - \eps \iint e^{(f(x) + g(y) - \frac{1}{2}\|x-y\|^2)/\eps}\dd P(x)\dd Q(y) + \eps\,.
\end{align}
%Note that the ``hard" constraint present in \cref{eq: ot_dual} is now replaced by a soft constraint in \cref{eq: dual_eot} on the right.
Both \cref{eq: primal_eot} and \cref{eq: dual_eot} possess solutions if $P$ and $Q$ have finite second moments; moreover, if we denote by $\pi_\eps$ the solution to \cref{eq: primal_eot}, which we call the \emph{optimal entropic plan}, and $(\feps, \geps)$ the solution to \cref{eq: dual_eot}, which we call the \textit{optimal entropic potentials}, then we obtain the optimality relation~\citep{Csi75}:
$$ \dd \pi_\eps(x,y) := \tilde{\pi}_\eps(x,y) \dd P(x) \dd Q(y) := \exp((f_\eps(x) + g_\eps(y) - \tfrac{1}{2}\|x-y\|^2)/\eps) \dd P(x)\dd Q(y)\,. $$
A consequence of this relation is that we may choose optimal entropic potentials satisfying
\begin{equation}\label{dual_opt}
\begin{split}
\int e^{\frac 1 \eps(f_\eps(x) + g_\eps(y) - \frac{1}{2}\|x-y\|^2)} \dd P(x) & =1 \quad \forall y \in \RR^d \\
\int e^{\frac 1 \eps(f_\eps(x) + g_\eps(y) - \frac{1}{2}\|x-y\|^2)} \dd Q(y) & =1 \quad \forall x \in \RR^d\,.
\end{split}
\end{equation}
We will therefore always assume that \cref{dual_opt} holds.
Conversely, if a pair of functions $(f_\eps, g_\eps)$ satisfies the duality conditions \cref{dual_opt} $P \otimes Q$ almost everywhere, then $\exp((f_\eps(x) + g_\eps(y) - \frac{1}{2}\|x-y\|^2)/\eps)$ is the $P \otimes Q$ density of the optimal entropic plan.

Our proofs rely on a modified version of the duality relation given in \cref{eq: dual_eot}, in which the supremum is taken over a larger set of functions.
Though it is a straightforward consequence of Fenchel's inequality, we have not encountered this statement explicitly in the literature, so we highlight it here.
\begin{proposition}\label{prop:new_dual}
Assume $P$ and $Q$ possess finite second moments, and let $\pi_\ep$ be the optimal entropic plan for $P$ and $Q$.
Then
\begin{equation}\label{eq:new_dual}
S_\ep(P, Q) = \sup_{\eta \in L^1(\pi_\ep)} \int \eta \dd \pi_\ep - \ep \iint e^{(\eta(x, y) - \frac 12 \|x - y\|^2)/\ep} \dd P(x) \dd Q(y)  + \ep\,.
\end{equation}
\end{proposition}
Comparing this proposition with \cref{eq: dual_eot}, we see that we can always take $\eta(x, y) = f(x) + g(y)$, in which case \cref{eq:new_dual} reduces to \cref{eq: dual_eot}.
The novelty in \cref{prop:new_dual} therefore arises in showing that the quantity on the right side of \cref{eq:new_dual} is still bounded above by $S_\ep(P, Q)$.
We give the short proof of \cref{prop:new_dual} in \cref{sec: omitted_proofs}.

Several recent works have bridged the regularized and unregularized optimal transport regimes, with particular interest in the setting where $\eps \to 0$. Convergence of $\pi_\eps$ to $\pi_0$ was studied by \cite{carlier2017convergence} and \cite{Leo12}, and recent work has quantified the convergence of the plans \citep{ghosal2021stability,bernton2021entropic,klatt2020empirical,hundrieser2021limit} and the potentials \citep{nutz2021entropic, altschuler2021asymptotics,rigollet2022sample,masud2021multivariate} in certain settings. Convergence of $S_\eps(P,Q)$ to $\frac 12 W_2^2(P,Q)$ has attracted significant research interest: under mild conditions, \cite{pal2019difference} proves a first-order convergence result for general convex costs (replacing $\frac 12 \|\cdot\|^2$), and a second order expansion was subsequently obtained by \cite{chizat2020faster} and \cite{conforti2021formula}. We rely on the following bound which we provide a short proof of in \Cref{sec: sec-order-est}.
\begin{theorem}\label{thm: faster_thm} Suppose $P$ and $Q$ have bounded densities with compact support. Then
\begin{align}\label{eq: tamanini_main}
S_\eps(P,Q) - \frac{1}{2}W^2_2(P,Q)  + \eps \log((2 \pi \ep)^{d/2}) \leq - \frac{\eps}{2}\pran{\Ent(P) + \Ent(Q)} + \frac{\eps^2}{8}I_0(P,Q)\,,
\end{align}
where \!$I_0(P,Q)$ is the integrated Fisher information along the Wasserstein geodesic between the source measure $P$ and target measure $Q$.
\end{theorem}

\section{Estimator and main results}\label{sec: motivation}
Given the optimal entropic plan $\pi_\eps$ between $P$ and $Q$, we define its barycentric projection to be
\begin{equation}
 T_\eps(x):= \int y \dd \pi_{\eps}^x(y) = \E_{\pi_{\eps}}[Y \mid X = x]\,.
\end{equation}
\textit{A priori}, this map is only defined $P$-almost everywhere, making it unsuitable for evaluation outside the support of $P$.
In particular, since we will study the barycentric projection obtained from the optimal entropic plan between empirical measures, this definition does not extend outside the sample points.
However, the duality relation~\cref{dual_opt} implies that we may define a version of the conditional density of $Y$ given $X = x$ for \emph{all} $x \in \RR^d$ by
$$\dd \pi_\eps^x(y) =  e^{\frac 1 \eps(f_\eps(x) + g_\eps(y) - \frac{1}{2}\|x-y\|^2)} \dd Q(y) = \frac{ e^{\frac 1 \eps(g_\eps(y) - \frac{1}{2}\|x-y\|^2)} \dd Q(y)}{\int e^{ \frac 1 \eps( g_\eps(y') - \frac{1}{2}\|x-y'\|^2)} \dd Q(y')}\,,$$
where $(f_\eps, g_\eps)$ are the optimal entropic potentials.
This furnishes an extension of $ T_\eps$ to all of $\RR^d$ by
\begin{equation}
T_\eps(x) \defeq \frac{\int y e^{\frac 1 \eps(g_\eps(y) - \frac{1}{2}\|x-y\|^2)} \dd Q(y)}{\int e^{\frac 1 \eps( g_\eps(y) - \frac{1}{2}\|x-y\|^2)} \dd Q(y)}\,.
\end{equation}
We call $T_\eps$ the \emph{entropic map} between $P$ and $Q$, though we stress that $(T_\eps)_\sharp P \neq Q$ in general.
This natural definition is motivated by the following observation, which shows that the entropic map can also be defined as the map obtained by replacing the optimal potential in Brenier's theorem by its entropic counterpart.

\begin{proposition}\label{prop: ent_brenier}
	Let $(f_\eps, g_\eps)$ be optimal entropic potentials satisfying \cref{dual_opt}, and let $T_\eps$ be the entropic map. Then $T_\eps = Id - \nabla f_\eps$.
\end{proposition}
\begin{proof}
%	We write $c(x, y) = \frac{1}{2}\|x-y\|^2$.
	\cref{dual_opt} implies 
$$ f_\eps(x) = -\eps\log \int e^{(g_\eps(y) - \frac{1}{2}\|x-y\|^2)/\eps} \dd Q(y)\,.$$
Taking the gradient of this expression yields
\begin{align*}
\nabla f_\eps(x) &= -\eps \frac{\int (-(x-y)/\eps)e^{(g_\eps(y) - \frac{1}{2}\|x-y\|^2)/\eps} \dd Q(y)}{\int e^{(g_\eps(y) - \frac{1}{2}\|x-y\|^2)/\eps} \dd Q(y)} \\
&= x - \frac{\int ye^{(g_\eps(y) - \frac{1}{2}\|x-y\|^2)/\eps} \dd Q(y)}{\int e^{(g_\eps(y) - \frac{1}{2}\|x-y\|^2)/\eps} \dd Q(y)} = x - T_\eps(x)\,.
\end{align*}
\end{proof}

We write $P_n = \frac 1n \sum_{i=1}^n \delta_{X_i}$ and $Q_n = \frac 1n \sum_{i=1}^n \delta_{Y_i}$ for the empirical distributions corresponding to the samples from $P$ and $Q$, respectively.
Our proposed estimator is $T_{\eps, (n, n)}$, the entropic map between $P_n$ and $Q_n$, which can be written explicitly as
\begin{equation}
T_{\eps, (n, n)}(x) = \frac{\frac 1n \sum_{i=1}^n Y_i e^{\frac 1 \eps(g_{\eps, (n,n)}(Y_i) - \frac 12 \|x -Y_i \|^2)}}{\frac 1n \sum_{i=1}^n e^{\frac 1 \eps(g_{\eps, (n,n)}(Y_i) - \frac 12 \|x -Y_i \|^2)}}\,,
\end{equation}
where $g_{\eps, (n, n)}$ is the optimal entropic potential corresponding to $Q_n$ in the optimal entropic plan between $P_n$ and $Q_n$, which can be obtained as part of the output of Sinkhorn's algorithm~\citep[see][]{PeyCut19}.
In other words, once the optimal entropic potential is found, the map $T_{\eps, (n, n)}(x)$ can therefore be evaluated in linear time. 
We discuss these computational aspects thoroughly in Section \ref{sec: exp}.
As in standard nonparametric estimation~\citep{Tsy09}, the optimal choice of $\eps$ will be dictated by the smoothness of the target function.

\begin{remark}
We briefly take a moment to discuss the applicability of our estimator in a wider statistical context. A body of work \citep[e.g.,][]{chernozhukov2017monge,hallin2021distribution} studies the estimation of multivariate ranks and quantiles through \emph{inverse} optimal transport maps. For this purpose, it is important that estimators of transport maps be invertible. We remark that the entropic map as defined above has this property since it is strongly monotone, in the sense that $(T_\eps(x) - T_\eps(y))^\top(x-y) > 0$ \citep[see][Proposition 10]{rigollet2022sample}. However, our procedure also gives rise to an even simpler estimator for the inverse transport map, namely the map $T_{\eps}^{\text{inv}} \defeq \text{id} - \nabla g_\eps$.
By interchanging the roles of $P$ and $Q$ in our assumptions, we can provide both computational and statistical guarantees for this map as well.
\end{remark}

To prove quantitative rates of convergence for $T_{\eps, (n, n)}$, we make the following regularity assumptions on $P$ and $Q$:
\begin{description}
\item \textbf{(A1)} $P, Q \in \cP_{ac}(\Omega)$ for a compact set $\Omega$, with densities satisfying $p(x), q(x) \leq M$ and $q(x) \geq m > 0$ for all $x \in \Omega$,
\item \textbf{(A2)} $\varphi_0 \in \cC^2(\Omega)$ and $\varphi_0^* \in \cC^{\alpha + 1}(\Omega)$ for $\alpha > 1$,
\item \textbf{(A3)}  $T_0 = \nabla \varphi_0$, with $\mu I \preceq \nabla^2 \varphi_0(x) \preceq L I$ for $\mu, L >0$ for all $x \in \Omega$,
\end{description}
In what follows, all constants may depend on the dimension, the set $\Omega$, $M$, $m$, $\mu$, $L$, 
and $\|\varphi_0^*\|_{\cC^{\alpha+1}}$. 

The above assumptions are qualitatively similar to those that have appeared in previous works on the estimation of optimal transport maps.
	
\textbf{(A1)} is a standard assumption in the statistical analysis of optimal transport map estimation. (It is present in, e.g., \cite{hutter2021minimax,manole2021plugin,deb2021rates,muzellec2021near}.)
All of these works require that $P$ and $Q$ be compactly supported.
Some of the tools we employ in this work extend beyond the compact support setting; for example, \cite{conforti2021formula} show that the expansion presented in \cref{thm: faster_thm} continues to hold for unbounded measures under suitable moment assumptions.
However, our proofs require strong \emph{a priori} bounds on the optimal transport map as well as on the entropic coupling for the random empirical measures $P_n$ and $Q_n$, which do not have clear analogues in the non-compact setting.

\textbf{(A3)} is also standard, and in prior work it has often been assumed implicitly as a consequence of a strengthened form of \textbf{(A1)}.
Caffarelli's regularity theory~\citep{Caffarelli1992} guarantees that if we assume that the set $\Omega$ in \textbf{(A1)} is \emph{convex} and that the density $p$ is also bounded below, then $T_0$ is continuous; if we further assume that $p, q \in C^\beta(\Omega)$ for any $\beta > 0$, then \textbf{(A3)} holds.
\textbf{(A3)} can therefore be viewed as being only slightly stronger than \textbf{(A1)}, so long as $\Omega$ is convex.
\textbf{(A3)} plays a crucial role in this and prior work, since, as was originally noticed by Ambrosio \citep[see][]{Gig11}, this assumption guarantees stability of the optimal transport map, as a function of the source and target measures.

Our most unusual assumption is \textbf{(A2)}.
Prior work analyzes estimators for $T_0$ under the assumption that $\varphi_0 \in \cC^{\alpha+1}(\Omega)$ for $\alpha > 1$, with rates that depend on $\alpha$.
For technical reasons, our proofs require a Laplace expansion in the ``target space'' corresponding to the dual Brenier potential $\varphi_0^*$.
Consequently, we instead assume that  $\varphi_0^* \in \cC^{\alpha+1}(\Omega)$, so that our rates depend on the smoothness of the \emph{inverse} map $T_0$.
We elaborate on this point further in the discussions surrounding  \Cref{lem: div_expansion}. 

Our main result is the following.
\begin{theorem}\label{thm:main}
	Under assumptions \textbf{(A1)} to \textbf{(A3)}, the entropic map $\hat T = T_{\eps, (n, n)}$ from $P_n$ to $Q_n$ with regularization parameter $\eps \asymp n^{-\frac{1}{d + \bar \alpha + 1 }}$ satisfies
	\begin{equation*}
	\E \|\hat T - T_0\|_{L^2(P)}^2 \lesssim (1 + I_0(P, Q)) n^{- \frac{(\bar \alpha + 1)}{2(d + \bar{\alpha} + 1)}} \log n\,,
	\end{equation*} 
	where $\bar \alpha = \alpha \wedge 3$.
\end{theorem}
When $d \to \infty$ and $\alpha \to 1$, we formally obtain the rate $n^{-(1+o(1))/d}$.
By contrast, \cite{hutter2021minimax} show that, up to logarithmic factors, the rate $n^{-2(1+o(1))/d}$ is minimax optimal in this setting.
\Cref{thm:main} therefore falls short of the minimax rate by a factor of approximately 2 in the exponent; however, our numerical experiments in \cref{sec: exp} show that $\hat T$ is competitive with minimax-optimal estimators in practice.

To analyze $\tse$, we adopt a two-step approach.
We first consider the one-sample setting and show that the entropic map $T_{\eps, n}$ between $P$ and $Q_n$ is close to $T_0$ in expectation.
We prove the following.
\begin{theorem}\label{thm:one_sample}
	Under assumptions \textbf{(A1)} to \textbf{(A3)} there exists a constant $\eps_0 >0$ such that for $\eps \leq \eps_0$, the entropic map $T_{\eps, n}$ between $P$ and $Q_n$ satisfies
	\begin{equation}
	\E \|T_{\eps, n} - T_0\|_{L^2(P)}^2 \lesssim \eps^{1-d/2} \log(n) n^{-1/2} + \eps^{(\bar \alpha + 1)/2} + \eps^2 I_0(P, Q)\,,
	\end{equation}
with  $\bar \alpha = \alpha \wedge 3$. Choosing $\eps \asymp n^{-\frac{1}{d + \bar \alpha - 1}}$, we get the one-sample estimation rate
\begin{equation}
\E \| T_{\eps,n} - T_0 \|^2_{L^2(P)} \lesssim (1 + I_0(P,Q))n^{-\frac{\bar \alpha + 1}{2(d + \bar \alpha - 1)}}\,. 
\end{equation}
\end{theorem}

\begin{remark}
	It can happen that $I_0(P, Q)$ is infinite, so the bounds of \cref{thm:main,thm:one_sample} are sometimes vacuous. However, \cite{chizat2020faster} prove that $I_0(P, Q) \leq C$ for a positive constant $C$ when $P$ and $Q$ satisfy \textbf{(A1)} to \textbf{(A3)} for $\alpha \geq 2$. Therefore, in this range for $\alpha$, we obtain the rates in the theorems above without additional restrictions.
\end{remark}

The proof of \cref{thm:one_sample} is technical, and our approach is closely inspired by \cite{pal2019difference} and empirical process theory arguments developed by \cite{GenChiBac18} and \cite{MenNil19}.
We give a summary of our argument here, and carry out the details in the following section. 

Following~\cite{pal2019difference}, we define the \emph{divergence} $D[y|x^*] := -x^\top y + \varphi_0(x) + \varphi_0^*(y)$, where $\varphi_0$ solves \cref{eq: semidual}.
Though this quantity is a function of $x$ and $y$, it is notationally convenient to write it in a way that highlights its dependence on $x^* := T_0(x)$.
Indeed, we rely throughout on the following fact
\begin{lemma}\label{lem: div_expansion}
Under assumptions \textbf{(A2)} and \textbf{(A3)}, for any $x \in \supp(P)$, we have
\begin{equation}\label{eq:d_expand}
D[y|x^*] = \frac 12 (y - x^*)^\top \nabla^2 \varphi_0^*(x^*) (y - x^*) + o(\|y - x^*\|^2) \quad \text{as $y \to x^*$}\,,
\end{equation}
as well as the non-asymptotic bound
\begin{equation}\label{quad_bounds}
\frac 1 {2L} \|y - x^*\|^2 \leq D[y|x^*] \leq \frac 1 {2\mu} \|y - x^*\|^2\,.
\end{equation}
\end{lemma}
\begin{proof}
This follows directly from Taylor's theorem and the fact that $\nabla \varphi_0^*(x^*) = T_0^{-1}(x^*) = x$.
\end{proof}

We then define a conditional probability density in terms of this divergence:
\begin{align}
\quad q_\eps^x(y) = \frac{1}{Z_\eps(x)\Lambda_\eps}e^{\frac{-1}{\eps}D[y|x^*]} \, , \quad Z_\eps(x) := \frac{1}{\Lambda_\eps}\int \exp\pran{-\frac{1}{\eps}D[y|x^*]} dy,
\end{align}
for $\Lambda_\eps = (2\pi\eps)^{d/2}$.
By virtue of~\cref{eq:d_expand}, if $\varphi_0^*$ is sufficiently smooth, then $q_\eps^x$ will be approximately Gaussian with mean $x^*$ and covariance $\eps \nabla^2 \varphi_0^*(x^*)^{-1} = \eps \nabla^2 \varphi_0(x)$.
We quantify this approximation via Laplace's method; details appear in \cref{sec: laplace_method}.
Using variational arguments, reminiscent of those employed by \cite{BobGot99} in the study of transportation inequalities, we then compare the measure $\pi_{\eps, n}$ to the measure  $q_\eps^x(y) \dd y\dd P(x)$ and compute accurate estimates of $T_{\eps, n}$ via Laplace's method.

A similar but much simpler argument establishes the following bound in the two-sample case.

\begin{theorem}\label{thm:two_sample}
	Let $T_{\eps, (n, n)}$ be the entropic map from $P_n$ to $Q_n$, and let $T_{\eps, n}$ be as in \cref{thm:one_sample}. Under assumptions \textbf{(A1)} to \textbf{(A3)}, for $\eps \leq 1$, $T_{\eps,(n,n)}$ satisfies
	\begin{equation*}
	\E \|T_{\eps, (n, n)} - T_{\eps, n}\|_{L^2(P)}^2 \lesssim \eps^{-d/2} \log (n) n^{-1/2}\,.
	\end{equation*}
\end{theorem}

Combining \cref{thm:one_sample,thm:two_sample} yields our main result.
\begin{proof}[Proof of \cref{thm:main}]
	We have
\begin{align*}
\E\| T_{\eps,(n,n)} - T_0\|^2_{L^2(P)} &\lesssim \E \| T_{\eps,(n,n)} - T_{\eps,n}\|^2_{L^2(P)} + \E\| T_{\eps,n} - T_0\|^2_{L^2(P)} \\
&\lesssim  \eps^{-d/2}\log(n)n^{-1/2} + \eps^{(\bar \alpha + 1)/2} + \eps^2 I_0(P, Q)\,.
\end{align*}
Choosing $\eps \asymp n^{- \frac{1}{d + \bar \alpha + 1}}$ yields the bound.
\end{proof}

\section{One-sample estimates}\label{sec: one_samp}
In this section, we prove \cref{thm:one_sample}, which relates $T_0$ to the entropic map between $P$ and $Q_n$:
\begin{equation*}
T_{\eps, n}(x)  = \frac{\int y e^{\frac 1 \eps(g_{\eps,n }(y) - \frac{1}{2}\|x-y\|^2)} \dd Q_n(y)}{\int e^{\frac 1 \eps( g_{\eps, n}(y) - \frac{1}{2}\|x-y\|^2)} \dd Q_n(y)}= \int y \dd \pi_{\eps, n}^x(y)\,,
\end{equation*}
where $\pi_{\eps, n}$ is the optimal entropic plan for $P$ and $Q_n$.
We stress that since $T_{\eps, n}$ is based on the entropic map from $P$ to $Q_n$, the second equality holds for $P$-almost every $x$.

Our main tool is the following inequality, which allows us to compare $\pi_{\eps, n}$ to the measure constructed from the conditional densities $q_\eps^x$.
The proof relies crucially on \cref{prop:new_dual} and on the second order-expansion provided in \cref{thm: faster_thm}.
%We provide a proof in \cref{sec: omitted_proofs}.

\begin{proposition}\label{prop: main}
Assume \textbf{(A1)} to \textbf{(A3)}, and let $a \in [L \eps, 1]$ for $\eps \leq 1$.
Then
\begin{multline}
\E \Big\{\sup_{h: \Omega \to \RR^d} \iint \left(h(x)^\top(y - T_0(x))  - a \|h(x)\|^2\right) \dd \pi_{\eps, n}(x, y) \\ - \iint (e^{h(x)^\top(y - T_0(x))  - a \|h(x)\|^2} - 1)q_\eps^x(y)\dd y \dd P(x) \Big\} \\ \lesssim \eps I_0(P,Q) + \eps^{(\bar \alpha - 1)/2} + \eps^{-d/2} \log(n) n^{-1/2}\,,
\end{multline}
where the supremum is taken over all $h \in L^2(P)$.
\end{proposition}
\begin{proof}
	Given $h \in L^2(P)$, write
	\begin{equation*}
	j_h(x, y) = h(x)^\top (y - T_0(x)) - a \|h(x)\|^2\,.
	\end{equation*}
	Choosing $\eta(x, y) = \eps(j_h(x, y) + \log(q_\eps^x(y)/q(y)) )+ \|x- y\|^2/2$ and applying \cref{prop:new_dual} with the measures $P$ and $Q_n$, we obtain
	\begin{align*}
	\sup_{h: \Omega \to \RR^d} \int j_h \dd \pi_{\eps, n} &+ \int \log\frac{q_\eps^x(y)e^{\frac 1{2\eps} \|x- y\|^2}}{q(y)} \dd \pi_{\eps, n}(x, y) \\
	&- \iint e^{j_h(x, y)} \frac{q_\eps^x(y)}{q(y)} \dd Q_n(y) \dd P(x) + 1 \leq \eps^{-1} S_\eps(P, Q_n)\,.
	\end{align*}
	
	We first expand $\iint \log\frac{q_\eps^x(y)e^{\frac 1{2\eps} \|x- y\|^2}}{q(y)}\dd \pi_{\eps,n}(x,y)$, where we use the fact that $\pi_{\eps,n}$ has marginals $P$ and $Q_n$:
	\begin{align*}
	&\iint \log\frac{q_\eps^x(y)e^{\frac 1{2\eps} \|x- y\|^2}}{q(y)}\dd \pi_{\eps,n}(x,y) \\
	&\qquad = \frac 1 \eps \iint \bigg[f_0(x) + g_0(y) + \eps\log\pran{\frac{1}{Z_\eps(x)\Lambda_\eps}} - \eps\log(q(y))\bigg] \dd \pi_{\eps,n}(x,y) \\
	&\qquad = \frac 1 \eps \Big(\int f_0(x)\dd P(x) + \int g_0(y)\dd Q_n(y)\Big) - \log(\Lambda_\eps) \\
	&\qquad \phantom{=} -  \int \log(Z_\eps(x)) \dd P(x) - \int \log(q(y))\dd Q_n(y)\,,
	\end{align*}
where $(f_0, g_0)$ solve \eqref{eq: ot_dual}. Replacing $Q_n$ by $Q$ yields
\begin{align*}
	\iint \log\frac{q_\eps^x(y)e^{\frac 1{2\eps} \|x- y\|^2}}{q(y)}\dd \pi_{\eps,n}(x,y) & = \frac{1}{2\eps }W_2^2(P,Q)  - \log(\Lambda_\eps)
	-  \int \log(Z_\eps(x)) \dd P(x)  \\
	& \phantom{=}- \Ent(Q)+ \int (g_0/\eps - \log(q))(\rd Q_n - \rd Q).	
\end{align*}

	A change of variables~\citep[see][Lemma 3(iv)]{pal2019difference}
	implies
	$$ \frac{\Ent(Q) - \Ent(P)}{2} = \int \log J(\nabla^2 \varphi_0^*(x^*)) \dd P(x)\,, $$
	where we recall that $x^* = T_0(x)$.
	Substituting this identity into the preceding expression yields
	\begin{align*}
	\iint \log\frac{q_\eps^x(y)e^{\frac 1{2\eps} \|x- y\|^2}}{q(y)}\dd \pi_{\eps,n}(x,y) &= \frac{1}{2\eps }W_2^2(P,Q)  - \log(\Lambda_\eps) - \frac{1}{2}(\Ent(Q)+\Ent(P))\\
	& \phantom{=} + \int (g_0/\eps-\log(q))(\rd Q_n - \rd Q) \\
	&\phantom{=} - \int \log(Z_\eps(x) J(\nabla^2 \varphi_0^*(x^*))) \dd P(x) \,.
	\end{align*}
	
	We therefore obtain
	\begin{multline*}
	\sup_{h: \Omega \to \RR^d} \int j_h \dd \pi_{\eps, n}  - \iint e^{j_h(x, y)} \frac{q_\eps^x(y)}{q(y)} \dd Q_n(y) \dd P(x) + 1 \\ \leq \eps^{-1} \Big(S_\eps(P, Q_n) - \frac 12 W_2^2(P, Q) + \eps \log(\Lambda_\eps) + \frac\eps 2 (\Ent(Q) + \Ent(P))\Big) + \Delta_1\,,
	\end{multline*}
	where $\Delta_1 := \int (g_0/\eps-\log(q))(\rd Q - \rd Q_n) + \int \log(Z_\eps(x) J(\nabla^2 \varphi_0^*(x^*))) \dd P(x)$.
	Applying \cref{thm: faster_thm}, we may further bound
	\begin{equation*}
	\sup_{h: \Omega \to \RR^d} \int j_h \dd \pi_{\eps, n}  - \iint e^{j_h(x, y)} \frac{q_\eps^x(y)}{q(y)} \dd Q_n(y) \dd P(x) + 1 \leq \frac \eps 8 I_0 + \Delta_1 + \Delta_2\,,
	\end{equation*}
	where $\Delta_2 := \eps^{-1}(S_\eps(P, Q_n) - S_\eps(P, Q))$.
	Now we turn our attention to the second term on the left side.
	Since
	\begin{align*}
	\iint e^{j_h(x, y)} \frac{q_\eps^x(y)}{q(y)} \dd Q(y) \dd P(x) &= \iint_{\supp(Q)} e^{j_h(x, y)}q_\eps^x(y) \dd y \dd P(x) \\
	&\leq \iint e^{j_h(x, y)}q_\eps^x(y) \dd y \dd P(x)\,,
	\end{align*}
	we have
	\begin{equation*}
	\sup_{h: \Omega \to \RR^d} \int j_h \dd \pi_{\eps, n}  - \iint (e^{j_h(x, y)} - 1)q_\eps^x(y) \dd y \dd P(x) \leq \frac \eps 8 I_0 + \Delta_1 + \Delta_2 + \Delta_3\,,
	\end{equation*}
	where
	\begin{equation*}
	\Delta_3 := \sup_{h: \Omega \to \RR^d} \iint e^{j_h(x, y)} \frac{q_\eps^x(y)}{q(y)} \dd P(x) (\rd Q_n - \rd Q)(y) 
	\end{equation*}
	and where we have used the fact that $q_\eps^x(y)$ is a probability density.
	
	It therefore remains only to show that
	\begin{equation*}
	\E [\Delta_1 + \Delta_2 + \Delta_3] \lesssim \eps^{(\bar \alpha - 1)/2} + \eps^{-d/2} \log(n) n^{-1/2}\,.
	\end{equation*}
	
	First, a Laplace expansion (\cref{laplace_z}) implies
	\begin{equation*}
	\E \Delta_1 = \int \log(Z_\eps(x) J(\nabla^2 \varphi_0^*(x^*))) \dd P(x) \lesssim \eps^{(\bar \alpha - 1)/2}
	\end{equation*}
	Second, known results on the finite-sample convergence of the Sinkhorn divergence (\cref{cor: emp_proc_special}) yield
	\begin{equation*}
	\E \Delta_2 \lesssim (\eps^{-1} + \eps^{-d/2}) \log(n) n^{-1/2}\,,
	\end{equation*}
	
	It therefore remains to bound $\Delta_3$, which an empirical process theory argument (\cref{exp_empirical_process}) shows
	\begin{equation*}
	\E \Delta_3 \lesssim \eps^{-d/2} n^{-1/2}
	\end{equation*}
	as long as $a \in [L\eps, 1]$.
	
	We obtain that
	\begin{equation*}
	\E [\Delta_1 + \Delta_2 + \Delta_3] \lesssim \eps^{(\bar \alpha - 1)/2} + (\eps^{-1} + \eps^{-d/2}) \log(n) n^{-1/2} + \eps^{-d/2}  n^{-1/2}\,,
	\end{equation*}
	and since $\eps \leq 1$, we obtain the bound
	\begin{equation*}
	\E [\Delta_1 + \Delta_2 + \Delta_3] \lesssim \eps^{(\bar \alpha - 1)/2} + \eps^{-d/2}  n^{-1/2}\log(n)\,,
	\end{equation*}
	as desired.
\end{proof}

To exploit \cref{prop: main}, we show that we can choose a function $h$ for which the left side of the above expression scales like $\|T_{\eps, n} - T_0\|_{L^2(P)}^2$.

We first establish three lemmas, whose proofs are deferred.
\begin{lemma}\label{lem:subg}
Fix $x \in \supp(P)$, and write $\bar y^x = \int y q_\eps^x(y) \dd y$.
There exists a positive constant $C$, independent of $x$, such that for all $\eps  \in (0, 1)$ and $\|v\|_2 \leq 1$,
\begin{equation*}
\int e^{(v^\top(y - \bar y^x))^2/(C\eps)} q_\eps^x(y) \dd y \leq 2\,.
\end{equation*}
\end{lemma}

In probabilistic language, \Cref{lem:subg} implies that if $Y^x$ is a random variable with density $q_\eps^x$, then $\eps^{-1/2}(Y^x - \E Y^x)$ is subgaussian~\citep{Ver18}.
By applying standard moment bounds for subgaussian random variables, we then arrive at the following result.
\begin{lemma}\label{lem:exp_bound}
There exists a positive constant $C$ such that if $a \geq C \ep$, then for any $h: \RR^d \to \RR^d$ we have
\begin{equation*}
\iint e^{h(x)^\top(y - T_0(x))  - a \|h(x)\|^2} q_\eps^x(y) \dd y \dd P(x) \leq \int e^{\frac 1 {4\eps} \|\bar y^x - T_0(x)\|^2} \dd P(x)\,.
%\iint (h(x)^\top(y - T_0(x))  - a \|h(x)\|^2) \dd \pi_{\eps, n}(x, y) \\ - \iint e^{h(x)^\top(y - T_0(x))  - a \|h(x)\|^2} q_\eps^x(y) \dd y \dd P(x) + 1
\end{equation*}
\end{lemma}

Finally, we show by an application of Laplace's method that $\bar y^x$ is close to $T_0(x)$.
\begin{lemma}\label{lem:mean}
Assume \textbf{(A1)} to \textbf{(A3)}. For all $x \in \supp(P)$,
\begin{equation*}
\|\bar y^x - T_0(x)\|^2 \lesssim \eps^{\alpha \wedge 2}\,.
\end{equation*}
\end{lemma}

With these lemmas in hand, we can complete the proof.
\begin{proof}[Proof of \cref{thm:one_sample}]
We may assume $\eps_0 \leq 1$.
Since $e^t -1 \leq 2t$ for $t \in [0, 1]$, \cref{lem:mean} implies that as long as $\eps_0$ is sufficiently small, for $\eps \leq \eps_0$,
\begin{equation*}
e^{\frac 1{4\eps} \|\bar y^x - T_0(x)\|^2} - 1 \lesssim \eps^{(\alpha - 1) \wedge 1} \leq \ep^{(\bar{\alpha} - 1)/2}\,,
\end{equation*}
where the last inequality holds for $\alpha \geq 1$ and $\eps \leq 1$. Combining this fact with \cref{lem:exp_bound}, we obtain that for any $h: \RR^d \to \RR^d$ and $a \geq C\eps$,
\begin{equation*}
\iint (e^{h(x)^\top(y - T_0(x))  - a \|h(x)\|^2} -1)q_\eps^x(y) \dd y \dd P(x)  \lesssim \ep^{(\bar{\alpha} - 1)/2}\,.
\end{equation*}

For a sufficiently small constant $\eps_0$, the interval $[C \eps, 1]$ is non-empty for $\eps \leq \eps_0$, so combining this fact with \cref{prop: main} yields that for $a \in [C \eps, 1]$ and $\eps \leq \eps_0$,
%\begin{equation}
%\begin{split}
\begin{multline}\label{eq:dual_with_h}
\E \sup_{h: \Omega \to \RR^d} \iint \left(h(x)^\top(y - T_0(x))  - a \|h(x)\|^2\right) \dd \pi_{\eps, n}(x, y) \lesssim \eps I_0 + \eps^{(\bar \alpha - 1)/2} 
+ \eps^{-d/2} \log(n) n^{-1/2}\,.
\end{multline}

%\end{split}
%\end{equation}
If we pick $h(x) = \frac 1{2 a} (T_{\eps, n}(x) - T_0(x))$, the integral on the left side equals
\begin{multline}\label{eq:h_choice}
\frac{1}{2 a} \E \iint \left((T_{\eps, n}(x) - T_0(x))^\top(y - T_0(x))  - \frac 12 \|T_{\eps, n}(x) - T_0(x)\|^2\right) \dd \pi_{\eps, n}(x, y)
\end{multline}
By definition, $T_{\eps, n}(x) = \int y \dd\pi_{\eps, n}^x(y)$, so disintegrating $\pi_{\eps, n}(x, y)$ and recalling that the first marginal of $\pi_{\eps, n}$ is $P$ yields
\begin{multline*}
\iint \left(T_{\eps, n}(x) - T_0(x))^\top(y - T_0(x)) - \frac 12 \|T_{\eps, n}(x) - T_0(x)\|^2\right)\dd \pi_{\eps, n}(x, y) \\ = \int \frac 12 \|T_{\eps, n}(x) - T_0(x)\|^2 \dd P(x) = \frac 12 \|T_{\eps, n} - T_0\|_{L^2(P)}^2\,.
\end{multline*}
Combining this with \cref{eq:dual_with_h} and \cref{eq:h_choice} and picking $a = C\eps$ yields
\begin{equation*}
\E \|T_{\eps, n} - T_0\|_{L^2(P)}^2 \lesssim \eps^2 I_0 + \eps^{(\bar \alpha + 1)/2} + \eps^{1-d/2} \log(n) n^{-1/2}\,,
\end{equation*}
as desired.
\end{proof}

As a corollary to \cref{thm:one_sample}, we have the following population-level estimate between $T_\eps$ and $T_0$, which is potentially of independent interest.
\begin{corollary}\label{cor: approx_error_bound}
Assume \textbf{(A1)} to \textbf{(A3)}, then
\begin{align}
\|T_\eps - T_0\|^2_{L^2(P)} = \| \nabla f_\eps - \nabla f_0\|^2_{L^2(P)} \lesssim \eps^2I_0(P,Q) + \eps^{(\bar{\alpha}+1)/2},
\end{align}
where $\bar{\alpha} = 3 \wedge \alpha$.
\end{corollary}

\section{Two-sample estimates}\label{sec: two_samp}
We now turn our attention to the two-sample case.
Let $\pi_{\eps, (n, n)}$ be the optimal entropic plan between $P_n$ and $Q_n$ and $(f_{\eps, (n, n)}, g_{\eps, (n, n)})$ the corresponding entropic potentials.
We aim to show that
\begin{equation*}
\E \|T_{\eps, (n, n)} - T_{\eps, n}\|_{L^2(P)}^2 \lesssim  (\eps^{-1} + \eps^{-d/2}) \log(n) n^{-1/2}\,.
\end{equation*}

As in \cref{sec: one_samp}, we proceed via duality arguments, but our task is considerably simplified by the fact that the measure $Q_n$ remains fixed in passing from $T_{\eps, (n, n)}$ to $T_{\eps, n}$.
Let us write
\begin{equation*}
\gamma(x, y) = e^{\frac 1 \eps (f_{\eps, (n, n)}(x) + g_{\eps, (n, n)}(y) - \frac 12 \|x - y\|^2)} = \frac{e^{\frac 1 \eps (g_{\eps, (n, n)}(y) - \frac 12 \|x - y\|^2)}}{\frac 1n \sum_{i=1}^n e^{\frac 1 \eps (g_{\eps, (n, n)}(Y_i) - \frac 12 \|x - Y_i\|^2)}}
\end{equation*}
for the $P_n \otimes Q_n$ density of $\pi_{\eps, (n, n)}$, where the second equality holds $P_n \otimes Q_n$ almost everywhere and furnishes an extension of $\gamma$ to all $x \in \RR^d$.

We employ the following analogue of \cref{prop: main}, which does not require the full force of assumptions \textbf{(A1)} to \textbf{(A3)}.
\begin{proposition}\label{prop:two_sample_main}
The support of $P$ and $Q$ lies in $\Omega$, then
\begin{align*}
\E \Big\{\sup_{\chi: \Omega \times \Omega \to \RR} \iint \chi(x, y) \dd \pi_{\eps, n}(x, y) &- \iint (e^{\chi(x, y)} - 1) \gamma(x, y) \dd P(x) \dd Q_n(y)\Big\} \\
&\lesssim (\eps^{-1} + \eps^{-d/2}) \log(n) n^{-1/2}\,,
\end{align*}
where the supremum is taken over all $\chi \in L^1(\pi_{\ep, n}).$
\end{proposition}
The proof of \cref{thm:two_sample} is now straightforward.
\begin{proof}
As in the proof of \cref{thm:one_sample}, consider
\begin{equation*}
\chi(x, y) = h(x)^\top(y - T_{\eps, (n, n)}(x)) - a \|h(x)\|^2
\end{equation*}
for $h$ and $a$ to be specified.
By definition of $T_{\eps, (n,n)}$, we have
%\begin{align*}
\begin{multline*}
\int h(x)^\top(y - T_{\eps, (n, n)}(x)) \gamma(x, y) \dd Q_n(y) = h(x)^\top \left(\int y \gamma(x,y)\dd Q_n(y) - T_{\eps,(n,n)}(x)\right)\\
= h(x)^\top \left(\frac{\frac 1n \sum_{i=1}^n  Y_i e^{\frac 1 \eps(g_{\eps, (n,n)}(Y_i) - c(x, Y_i))}}{\frac 1n \sum_{i=1}^n e^{\frac 1 \eps(g_{\eps, (n, n)}(Y_i) - c(x, Y_i))}}  - T_{\eps, (n, n)}(x)\right) = 0
\end{multline*}
for all $x \in \RR^d$.
Moreover, since $\Omega$ is compact, by the Cauchy-Schwarz inequality, there exists a constant $C$ such that
\begin{equation*}
|h(x)^\top (y - T_{\eps, (n, n)}(x))| \leq C \|h(x)\| \quad \forall y \in \Omega\,.
\end{equation*}
Hoeffding's inequality therefore implies that if $a \geq C^2/2$, then this choice of $\chi$ satisfies
\begin{equation*}
\iint (e^{\chi(x, y)} - 1) \gamma(x, y) \dd Q_n(y) \dd P(x) \leq 0\,.
\end{equation*}

Choosing $h(x) = \frac{1}{2a} (T_{\eps, n}(x) - T_{\eps, (n, n)}(x))$, we conclude as in the proof of \cref{thm:one_sample} that for $\eps \leq 1$,
\begin{equation*}
\frac{1}{4 a} \E \|T_{\eps, n} - T_{\eps, (n, n)}\|_{L^2(P)}^2 \lesssim (\eps^{-1} + \eps^{-d/2}) \log(n) n^{-1/2} \lesssim \eps^{-d/2} \log(n) n^{-1/2}\,,
\end{equation*}
and picking $a$ to be a sufficiently large constant yields the claim.
\end{proof}
%For $x \in \RR^d$, let us denote by $\pi_{\ep, n}^x$ the probability measure with $Q$-density proportional to $e^{\frac 1 \ep (\gepsn(y) - \frac 12 \|x - y\|^2)}$, and $\pi_\ep^x$ the probability measure with $Q$-density proportional to $e^{\frac 1 \ep (\geps(y) - \frac 12 \|x - y\|^2)}$.
%Then $\Tepsn(x)$ and $\Teps(x)$ are nothing but the mean of $\pi_{\ep, n}^x$ and $\pi_{\ep}^x$, respectively. Finally, let us define the $\ep$-conjugate $\conj{g}$ of a function $g \in L_1(Q)$ by
%\begin{equation*}
%\conj{g}(x) \defeq - \ep \log \int e^{\frac 1 \ep (g(y) - \frac 12 \|x - y\|^2)} \dd Q(y)\,.
%\end{equation*}

%We prove \cref{thm: main} by using a decomposition analogous to the bias-variance trade-off: 
%\begin{align}
%\mathbb{E}\| \hat{T}_{\eps,n} - T\|^2_{L^2(P)} \lesssim \E\| \hat{T}_{\eps,n}  - \bar{T}_{\eps,n} \|^2_{L^2(P)} + \E\| \bar{T}_{\eps,n} - \Teps\|^2_{L^2(P)} + \|T_\eps - T\|^2_{L^2(P)}\,.
%\end{align}

% 

 \section{Adaptive estimation}\label{sec: adaptive_estimation}
In Theorems \ref{thm:main} and \ref{thm:one_sample}, the optimal choice of the regularization parameter $\eps$ depends on $n$, $d$, and $\alpha$.
Although the number of samples and dimension are obviously known to the practitioner, the smoothness of the transport map is often not known \textit{a priori}. 
However, Lepski's method~\citep[see][]{birge2001alternative} can be used to obtain a data-driven method of choosing $\eps$, which gives rise to an estimator that adapts to the unknown smoothness parameter $\alpha$.
 
For notational convenience, for any $\alpha > 1$, let $s := \alpha + 1$ be the smoothness of the conjugate Brenier potential $\varphi_0^*$.
We assume that $s \in [2+\iota, 4]$ for some $\iota > 0$ sufficiently small and fixed.
Let $\cS$ be the following discrete subset
\begin{align*}
    \cS \defeq \{2+\iota = s_{\min} = s_1 < s_2 < \cdots < s_N = s_{\max} = 4\}\,,
\end{align*}
where $s_j - s_{j-1} \asymp (\log n)^{-1}$, and set
\begin{align}\label{eq: lepski_eps}
    \eps_s = (n/\log n)^{-1/2(d+s)}\,, \quad \psi_n(s) = (\eps_s)^s = (n/\log n)^{-s/2(d+s)}\,.
\end{align}
To calibrate our choice of $\eps$, we rely on sample splitting.
Let $\DD := \{(X_i,Y_i)\}_{i=1}^n$ denote our initial dataset, and let $\DD'$ denote an independent copy of $\DD$. Denote by $P_n'$ and $Q_n'$ the empirical measures arising from $\DD'$. Our choice of smoothness parameter is given by the following rule:
\begin{align}\label{eq: lepski_rule}
    \hat{s} \defeq \max\{ s \in \cS \ : \ \| \hat{T}_{\eps_{s}} - \hat{T}_{\eps_{s'}}\|_{L^2(P_n')}^2 \leq K \psi_n(s')\,, \forall \ s' \leq s, s' \in \cS \}\,,
\end{align}
for a positive constant $K$.
The following theorem shows that choosing $\eps = \eps_{\hat s}$ gives rise to an adaptive estimator.
\begin{theorem}\label{thm: adaptive_estimation}
Suppose \textbf{(A1)} to \textbf{(A3)} holds, with $X_1,\ldots,X_n \sim P$ and $Y_1,\ldots,Y_n \sim Q$, resulting in $\mathbb{D} = \{(X_i,Y_i)\}_{i=1}^{\lfloor n / 2 \rfloor}$ and a hold-out set $\mathbb{D}'$. Suppose $\hat{s}$ is chosen according to \cref{eq: lepski_rule} for $K$ sufficiently large, with $\eps = \eps_{\hat s}$ chosen as in \cref{eq: lepski_eps}. The resulting estimator $\hat{T}_{\eps_{\hat{s}}}$ exhibits a risk in $L^2(P)$ that matches \cref{thm:main} up to log factors.
\end{theorem}
The proof of \cref{thm: adaptive_estimation} uses standard ideas and is deferred to \Cref{sec: adaptivity}.

\section{Computational aspects}\label{sec: exp}
Our reason for studying the entropic map as an optimal transport map estimator arises from its strong computational benefits, which are a consequence of the efficiency of Sinkhorn's algorithm for entropic optimal transport~\citep[see][]{PeyCut19}. In this section, we compare the computational complexity of the entropic map to the estimators of \cite{hutter2021minimax}, \cite{deb2021rates}, and \cite{manole2021plugin} in the two-sample setting. Finally, we perform several experiments that demonstrate the computational advantages of our procedure. Throughout this section, we use $\tilde{O}$ to hide poly-logarithmic factors in the sample size $n$.
\subsection{Estimator complexities from prior work}\label{sec: prior_comp}
We first describe the wavelet-based estimator proposed by \cite{hutter2021minimax}.
Recall that this estimator is minimax optimal for all $\alpha > 1$.
The implementation of this estimator requires various discretization and approximation schemes.
The authors of that work use a numerical implementation of the Daubechies wavelets to approximate the optimal Brenier potential, and then compute its convex conjugate by means of a discrete Legendre transform on a discrete grid. The gradient of the resulting potential is then obtained using finite differences, and this is extended to data outside the grid by linear interpolation.
Though computing this estimator takes time that scales only linearly in the sample size $n$, the main bottleneck of this approach from a computational standpoint is the computation of the Legendre transform on the grid, which requires at least $c N^d$ operations, where $N$ denotes the resolution of the grid. Since this resolution needs to be chosen fine enough to be negligible, the exponential dependence in $d$ makes this approach prohibitively expensive in most applications.

Another estimator recently analyzed in the literature is the ``1-Nearest Neighbor'' estimator, which we denote by $\hat{T}^{\text{1NN}}_{(n,n)}$ \citep{manole2021plugin}, which achieves the minimax rate when $T_0$ is bi-Lipschitz (i.e., $\alpha = 1$ and \textbf{(A3)} is satisfied) over a compact domain $\Omega$.
The estimator takes the form
\begin{align}\label{eq: est_1nn}
\hat{T}^{\text{1NN}}_{(n,n)}(x) = \sum_{i,j=1}^n (n\hat{\pi}_{ij})\bm{1}_{V_i}(x)Y_j,
\end{align}
where $\bm{1}$ is the indicator function for a set, and $(V_i)_{i=1}^n$ are the Voronoi regions generated by $(X_i)_{i=1}^n$, i.e.
$$ V_i = \{ x \in \Omega \ : \ \|x - X_i\| \leq \|x - X_j\|, \ \forall j \neq i\}, $$
and $\hat{\pi}$ is the optimal coupling that solves \cref{eq: kant_p} when the measures are the empirical measures $P_n$ and $Q_n$. Solving for $\hat{\pi}$ can be done through the Hungarian algorithm, and has time complexity ${O}(n^3)$.
However, unlike the wavelet estimator described above, computing this estimator does not require constructing a grid whose size scales exponentially with dimension.

For the $\alpha > 1$ case, both \cite{manole2021plugin} and \cite{deb2021rates} propose estimators based on density estimation. For these approaches, the idea is to construct nonparametric density estimates of the measures $P$ and $Q$, resample points from these densities, and finally perform the appropriate matching using the Hungarian algorithm once again. Though tractable in low dimensions, this approach is limited by the difficulty of sampling from nonparametric density estimates, which typically requires time scaling exponentially in the dimension $d$. 

%\jnw{I'm inclined to omit this paragraph unless the reviewer insists we keep it in.}
%\textcolor{cyan}{[TO CHANGE OR OMIT]
%Finally, we discuss the kernel sum-of-squares estimator that arises in \cite{muzellec2021near}. In the high-smoothness regime, their proposed estimator is near-minimax optimal in the number of samples and the dimension, with the underlying constant degenerating exponentially in the dimension. While this seems potentially promising, their proposed algorithms do not coincide with their proven results: their simulations rely on a \textit{regularized} version of their sum-of-squares objective which allows them to take gradients of the potential functions. It is unclear if their proposed methodology is consistent with their theoretical findings, and for this reason, we omit algorithmic comparisons.
%}

In short, prior estimators proposed in the literature either have runtime scaling exponentially in $d$ (in the case of the wavelet estimator or estimators based on nonparametric density estimation) or cubicly in $n$ (in the case of the 1NN estimator).
By contrast, in the following section, we show that our estimator can be computed in nearly $O(n^2)$ time.

%We note that the analysis of \cite{manole2021plugin} is performed on the $d$-dimensional torus, which falls outside the scope of our estimator. While the work of \cite{deb2021rates} considers compact domains, they are not able to explicitly characterize the runtime of estimating $\hat{P}$ and $\hat{Q}$ (see Remark 2.9 in their paper). They show, however, that if one were able to sample from the estimated densities, their approach would require a sample complexity of $O(n^{\frac{3(s+2)}{2}})$ to achieve minimax accuracy, where $s$ parameterizes the smoothness of the densities. This result exemplifies the ``no free lunch principle" where, indeed, their method is theoretically optimal and runtime is independent of dimension, but scales very poorly in smoothness.
\subsection{Computational complexity of the Entropic Map}
We now turn to the computational analysis of our estimator, which has the closed-form representation
\begin{equation}\label{eq: tnn_est}
\hat{T}_{\eps, (n, n)}(x) = \frac{ \sum_{i=1}^n Y_i e^{\frac 1 \eps(g_{\eps, (n,n)}(Y_i) - \frac 12 \|x -Y_i \|^2)}}{ \sum_{i=1}^n e^{\frac 1 \eps(g_{\eps, (n,n)}(Y_i) - \frac 12 \|x -Y_i \|^2)}}\,.
\end{equation}
The computational burden of our estimator falls on computing the optimal entropic potential evaluated at the data $g_{\eps,(n,n)}(Y_i)$. Indeed, once we have this potential, it is clear that the remainder of \Cref{eq: tnn_est} can be computed in $O(n)$ time.

The leading approach to compute optimal entropic potentials in practice is \emph{Sinkhorn's algorithm}~\citep{PeyCut19, Sin67}, an alternating minimization algorithm that computes approximations of the entropic potentials by iteratively updating $f$ and $g$ so that they satisfy one of the two dual optimality conditions given in~\cref{dual_opt}.
Explicitly, defining $f^{(0)} = 0$, Sinkhorn's algorithm performs the updates
\begin{align*}
	g^{(k)}(y) &= - \eps \log \frac 1n \sum_{i=1}^n e^{\frac 1 \eps (f^{(k)}(X_i) - \frac 12 \|X_i - y\|^2)} \\
	f^{(k+1)}(x) &= - \eps \log \frac 1n \sum_{j=1}^n e^{\frac 1 \eps (g^{(k)}(Y_j) - \frac 12 \|x - Y_j\|^2)}\,.
\end{align*}
until termination.
Since it is only necessary to compute $f^{(k)}$ and $g^{(k)}$ on the support of $P_n$ and $Q_n$, respectively, each iteration can be implemented in $O(n^2)$ time.

Note that this update rule guarantees that
\begin{equation*}
	\int e^{\frac 1 \eps (f^{(k)}(x) + g^{(k)}(y) - \frac 12 \|x - y\|^2)} \dd P_n(x) = 1
\end{equation*}
for all $y$ at each iteration.
By contrast, the second optimality condition in~\cref{dual_opt} is \emph{not} satisfied at each iteration, though~\cite{Sin67} showed that
\begin{equation*}
	\int e^{\frac 1 \eps (f^{(k)}(x) + g^{(k)}(y) - \frac 12 \|x - y\|^2)} \dd Q_n(y) \to 1
\end{equation*}
as $k \to \infty$, and therefore that the iterates of Sinkhorn's algorithm converge to optimal entropic potentials.

To analyze the running time of our estimator, we will leverage recent analyses of the convergence rate of Sinkhorn's algorithm~\citep{Cut13,AltWeeRig17,DvuGasKro18} to explicitly quantify the error incurred by terminating after a finite number of steps.
For $k \geq 0$, we consider the entropic map estimator obtained after $k$ iterates of Sinkhorn's algorithm:
\begin{align}\label{eq: est_k}
T^{(k)}(x) = \frac{ \sum_{i=1}^n Y_i e^{\frac 1 \eps(g^{(k)}(Y_i) - \frac 12 \|x -Y_i \|^2)}}{ \sum_{i=1}^n e^{\frac 1 \eps(g^{(k)}(Y_i) - \frac 12 \|x -Y_i \|^2)}}\,.
\end{align}
Despite the fact that $g^{(k)}$ is \emph{not} an entropic potential for the original problem, the following theorem shows that $T^{(k)}$ is nevertheless an acceptable estimator if $k$ is sufficiently large.
%As long as the $k^{\text{th}}$ iterate is one that satisfies the left-marginal $Q_n$, we are able to provide quantitative lower bounds on the number of iterates needed to approximate the optimal transport map, up to our previously proven statistical errors.
%In summary, this procedure is known to estimate of $\pi_{\eps,(n,n)}$ in time $\tilde O(n^2)$.\footnote{Theoretical guarantees on the runtime of Sinkhorn's algorithm depend polynomially on the regularization parameter $\ep$; however, since we choose $\ep = n^{- O(1/d)}$, these factors do not significantly affect the empirical performance of the algorithm in high dimensions. \jnw{This footnote is no longer really needed now that we have the theorem below, since we will hopefully make explicit our running time guarantees.}}  The following result explicitly characterizes the runtime of our map estimator in terms of the regularization parameter. \jnw{Theorem does not mention runtime anywhere. Make explicit!}

\begin{theorem}\label{thm: comp_guarantee}
Suppose  assumptions \textbf{(A1)} to \textbf{(A3)} hold, and we choose $\eps$ as in \Cref{thm:main}.
Then for any $k \gtrsim n^{7/(d+\bar{\alpha}+1)} \log n$,
\begin{equation*}
\E \|T^{(k)} - T_0\|^2_{L^2(P)} \lesssim (1 + I_0(P, Q)) n^{- \frac{(\bar \alpha + 1)}{2(d + \bar{\alpha} + 1)}} \log n\,,
\end{equation*}
where $\bar \alpha = 3 \wedge \alpha$.
In particular, an estimator achieving the same rate as the estimator in \cref{thm:main} can be computed in $\tilde O(n^{2 + 7/(d+\bar{\alpha}+1)}) = n^{2 + o_d(1)}$ time.
\end{theorem}
\begin{proof}
%Without loss of generality, assume $\pi_{\eps,(n,n)}^{(k)}$ has left-marginal $Q_n$. If not, we increment the number of iterations by 1 such that this condition is satisfied.  Let $\tilde{P}_n$ denote the right-marginal of $\pi_{\eps,(n,n)}^{(k)}$, and in particular we can write 
%\begin{align*}
%\pi_{\eps,(n,n)}^{(k)}(x,y) = \tilde{\gamma}(x,y)\dd\tilde{P}_n(x)\dd Q_n(y)\,,
%\end{align*}
%where
%\begin{align*}
%\tilde{\gamma}(x,y) = \exp\pran{\frac{ f^{(k)}_{\eps,(n,n)}+ g_{\eps,(n,n)}^{(k)}-\tfrac12\|x-y\|^2 }{\eps}}\,.
%\end{align*}
We begin by decomposing the error and applying \Cref{thm:one_sample}:
\begin{align*}
\E \|T^{(k)} - T_0\|^2_{L^2(P)} &\lesssim \E \|T^{(k)} - T_{\eps,n}\|^2_{L^2(P)} + \E \|T_{\eps,n} - T_0\|^2_{L^2(P)} \\
&\lesssim  \E \|T^{(k)} - T_{\eps,n}\|^2_{L^2(P)} + \eps^{1-d/2} \log(n) n^{-1/2} + \eps^{(\bar \alpha + 1)/2} + \eps^2 I_0(P, Q)\,.
\end{align*}

We proceed almost exactly as in \Cref{thm:two_sample}, and consider
\begin{equation*}
\chi(x,y) = h(x)^\top\pran{y - {T}^{(k)}(x)} - a\|h(x)\|^2,
\end{equation*}
for $h: \RR^d \to \RR^d$ and $a$ to be specified.
	For $x \in \RR^d$, $y \in \supp(Q_n)$, define
\begin{equation}
\tilde{\gamma}(x,y) = \frac{\exp\pran{\frac 1 \eps (g^{(k)}(y)-\tfrac12\|x-y\|^2)}}{\frac 1n \sum_{i=1}^n \exp\pran{\frac 1 \eps(g^{(k)}(Y_i)-\tfrac12\|x-Y_i\|^2)}}\,.
\end{equation}

By construction, $\int \tilde \gamma(x, y) \dd Q_n(y) = 1$ for all $x \in \RR^d$, and $T^{(k)}(x) = \int y \tilde \gamma(x, y) \dd Q_n(y)$.
Therefore, for any $h: \RR^d \to \RR^d$,
\begin{equation*}
\int h(x)^\top\pran{y - {T}^{(k)}(x)} \tilde{\gamma}(x,y)\dd Q_n(y) = 0
\end{equation*}
for all $x \in \R^d$. Moreover, since $\Omega$ is compact, there exists a constant $C$ such that
\begin{align*}
|h(x)^\top(y - T^{(k)}(x))|\leq C\|h(x)\| \quad \forall x, y \in \Omega\,.
\end{align*}
Hoeffding's inequality therefore implies that for $a$ sufficiently large, this choice of $\chi$ satisfies
\begin{align*}
\iint (e^{\chi(x,y)} - 1)\tilde{\gamma}(x,y)\dd Q_n \dd P(x) \leq 0\,.
\end{align*}

Now, define a probability measure $\tilde P$ with the same support as $P_n$ by setting
\begin{equation}
	\frac{\rd \tilde P(x)}{\rd P_n(x)} = \int e^{\frac 1 \eps (f^{(k)}(x) + g^{(k)}(y) - \frac 12 \|x - y\|^2)} \dd Q_n(y)\,,
\end{equation}
and let
\begin{equation}
	f^{(k+1)}(x) = - \eps \log \frac 1n \sum_{i=1}^n  \exp\pran{\eps^{-1}(g^{(k)}(Y_i)-\tfrac12\|x-Y_i\|^2)}\,.
\end{equation}
We claim that $\tilde \gamma (x, y) = \exp(\frac 1\eps (f^{(k+1)}(x) + g^{(k)}(y) - \frac 12 \|x - y\|^2)$ is the $\tilde P \otimes Q_n$ density of the optimal entropic plan between $\tilde P$ and $Q_n$.
We have already observed that $\int \tilde \gamma(x, y) \dd Q_n(y) = 1$ for all $x \in \RR^d$ by construction, so it suffices to note that for all $y \in \supp(Q_n)$,
\begin{align*}
	\int \tilde \gamma(x, y) \dd \tilde P(x) & = \int \frac{e^{\eps^{-1}(g^{(k)}(y) - \frac 12 \|x - y\|^2)}}{\int e^{\eps^{-1}(g^{(k)}(y') - \frac 12 \|x - y'\|^2)} \dd Q_n(y')} \dd \tilde P(x) \\
	& = \int \frac{e^{\eps^{-1}(f^{(k)}(x) + g^{(k)}(y) - \frac 12 \|x - y\|^2)}}{\int e^{\eps^{-1}(f^{(k)}(x) + g^{(k)}(y') - \frac 12 \|x - y'\|^2)} \dd Q_n(y')} \dd \tilde P(x) \\
	& = \int e^{\eps^{-1}(f^{(k)}(x) + g^{(k)}(y) - c(x, y))} \dd P_n(x) = 1\,.
\end{align*}
Therefore $(f^{(k+1)}, g^{(k)})$ satisfy \cref{dual_opt}, so $\tilde \gamma$ is indeed the $\tilde P \otimes Q_n$ density of the optimal entropic plan between the two measures.

Applying \Cref{prop: delta_dev}, we obtain for any $\ep\leq 1$
\begin{equation}
	\E \sup_{h: \RR^d \to \RR^d} \iint h(x)^\top\pran{y - {T}^{(k)}(x)} - a\|h(x)\|^2 \dd \pi_{\ep, n} \lesssim \ep^{-1} \delta + \ep^{-d/2} \log(n) n^{-1/2}\,,
\end{equation}
where $\delta := \tv{\tilde{P}}{P_n}$.
Choosing $h(x) = \frac{1}{2a}\pran{T_{\eps,n}(x) - {T}^{(k)}(x)}$, we conclude as in  \Cref{thm:two_sample}, resulting in
\begin{align*}
\E \|T^{(k)} - T_{\eps,n}\|^2_{L^2(P)} \lesssim \eps^{-1}\delta + \eps^{-d/2}\log(n)n^{-1/2}\,.
\end{align*}
All together, we have 
\begin{align*}
\E \|T^{(k)} - T_0\|^2_{L^2(P)} &\lesssim  \eps^{-1}\delta +  \eps^{-d/2}\log(n)n^{-1/2} + \eps^{(\bar \alpha + 1)/2} + \eps^2 I_0(P, Q)\,.
\end{align*}
The first term will be negligible if $\delta \lesssim \eps^{3}$.

By definition, $\tilde P$ is the first marginal of the joint distribution with density $e^{\frac 1 \eps(f^{(k)}(x) + g^{(k)}(y) - \frac 12 \|x - y\|^2)}$ with respect to $P_n \otimes Q_n$.
By \citet[Theorem 2]{AltWeeRig17}, if $k$ satisfies
\begin{equation*}
	k \gtrsim \delta^{-2} \log(n \cdot \max_{i, j} e^{\frac{1}{2\eps}\|x_i - y_j\|^2}) \gtrsim \delta^{-2} \eps^{-1} \log n\,,
\end{equation*}
then $\tv{\tilde P}{P_n} \leq \delta$.
Choosing $\delta = \eps^{3} \asymp n^{-3/(d+\bar{\alpha}+1)}$ yields the claim.
\end{proof}

%The runtime of these three methods is summarized in \Cref{table: comp_compare}. It is worth noting that our map estimator is agnostic to the smoothness of the map, whereas $\hat{T}^{\text{1NN}}$ specifically caters for Lipschitz maps, meanwhile $\hat{T}^{\text{W}}$ is not analyzed in the Lipschitz case, and in principle should depend on the smoothness of the map.
\begin{remark}
		A surprising feature of \cref{thm: comp_guarantee} is that the necessary number of iterations decreases with the dimension $d$.
		This reflects the fact that when $d$ is large, the optimal choice of $\eps$ is also larger, and it is well established both theoretically and empirically that the performance of Sinkhorn's algorithm improves considerably as $\eps$ increases~\citep{AltWeeRig17,Cut13}.
\end{remark}

\subsection{Empirical performance}
We test two implementations of Sinkhorn's algorithm, one from the Python Optimal Transport (POT) library \citep{flamary2021pot}, and an implementation that uses the KeOps library optimized for GPUs. Both implementations employ log-domain stabilization to avoid numerical overflow issues arising from the small choice of $\eps$.

For simplicity, we employ the same experimental setup as~\cite{hutter2021minimax}.
We generate i.i.d.\ samples from a source distribution $P$, which we always take to be $[-1, 1]^d$, and from a target distribution $Q = (T_0)_\sharp P$, where we define $T_0: \RR^d \to \RR^d$ to be an optimal transport map obtained by applying a monotone scalar function coordinate-wise.\footnote{Note that any component-wise monotone function is the gradient of a convex function.}

In \Cref{fig: emp_ests_smalldim}, we visualize the output of our estimator in $d = 2$.
The figures depict the effect of evaluating the estimator $\hat{T}_\eps$ and the true map $T_0$ on additional test points $X_1', \dots, X'_m$ drawn i.i.d.\ from $P$.
\begin{figure}[t]
\centering
    \begin{subfigure}[H]{0.4\textwidth}            
            \includegraphics[width=\textwidth]{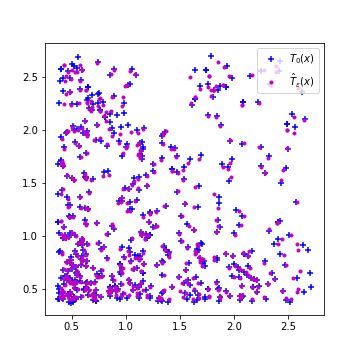}
            \caption{$T_0(x) = \exp(x)$ coordinate-wise}
    \end{subfigure}%
     \begin{subfigure}[H]{0.4\textwidth}
            \centering
            \includegraphics[width=\textwidth]{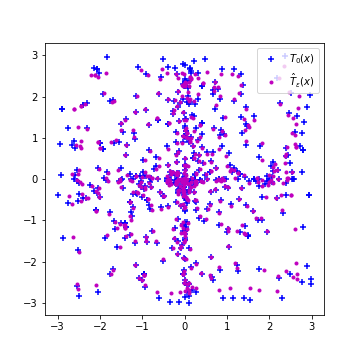}
            \caption{$T_0(x) = 3x^2\text{sign}(x)$ coordinate-wise}
    \end{subfigure}
    \caption{Visualization of $\hat{T}_\eps$ and $T_0(x)$ in 2 dimensions.}\label{fig: emp_ests_smalldim}
\end{figure}

\subsubsection{Comparison to a tractable minimax estimator}\label{sec: manole_estimator}
Among the previously discussed estimators, the 1-Nearest Neighbor estimator analyzed in \cite{manole2021plugin} is the most tractable, and the only one remotely comparable to our method. As discussed in \cref{sec: prior_comp}, this approach uses the Hungarian algorithm which has a runtime of $O(n^3)$. However, since it is not parallelizable, we compare its performance to the non-parallel CPU implementation of Sinkhorn's algorithm from the POT library.

We perform a simple experiment comparing our approach to theirs: let $P = [-1,1]^d$ and let $T_0(x) = \exp(x)$, acting coordinate-wise. We vary $d$ and $n$, and track runtime performance of both estimators, as well as the Mean Squared Error (MSE) of the map estimate\footnote{We calculate MSE by performing Monte Carlo integration over the space $[-1,1]^d$.}, averaged over 20 runs. For our estimator, we choose $\eps$ as suggested by \cref{thm:main}.
\begin{figure}[h]
\centering
    \begin{subfigure}[H]{0.5\textwidth}            
            \includegraphics[width=\textwidth]{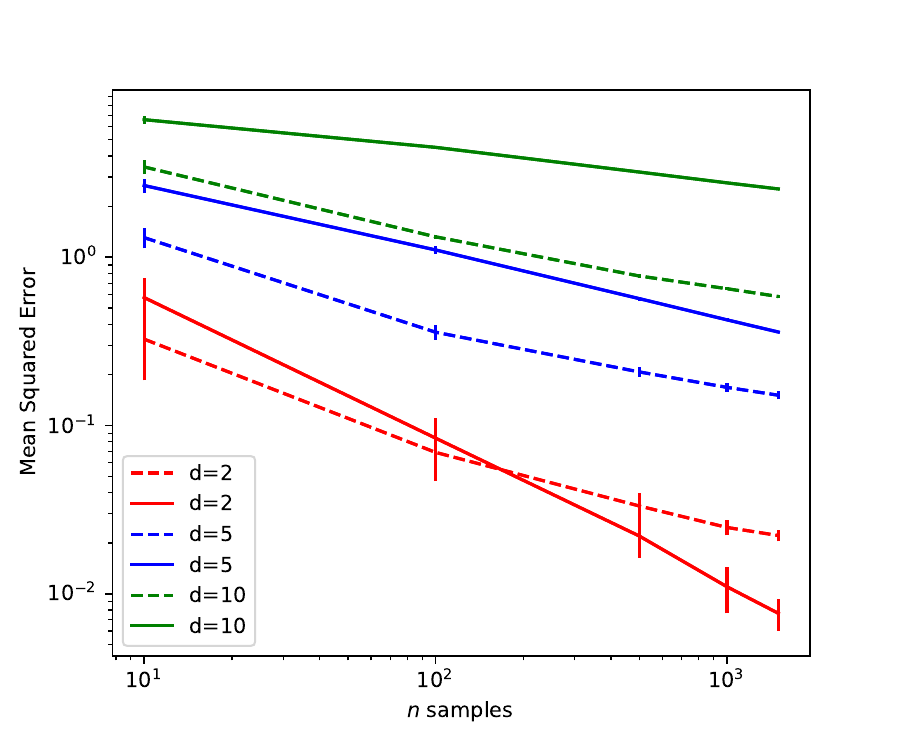}
            \caption{MSE comparison}
    \end{subfigure}%
    \begin{subfigure}[H]{0.5\textwidth}            
            \includegraphics[width=\textwidth]{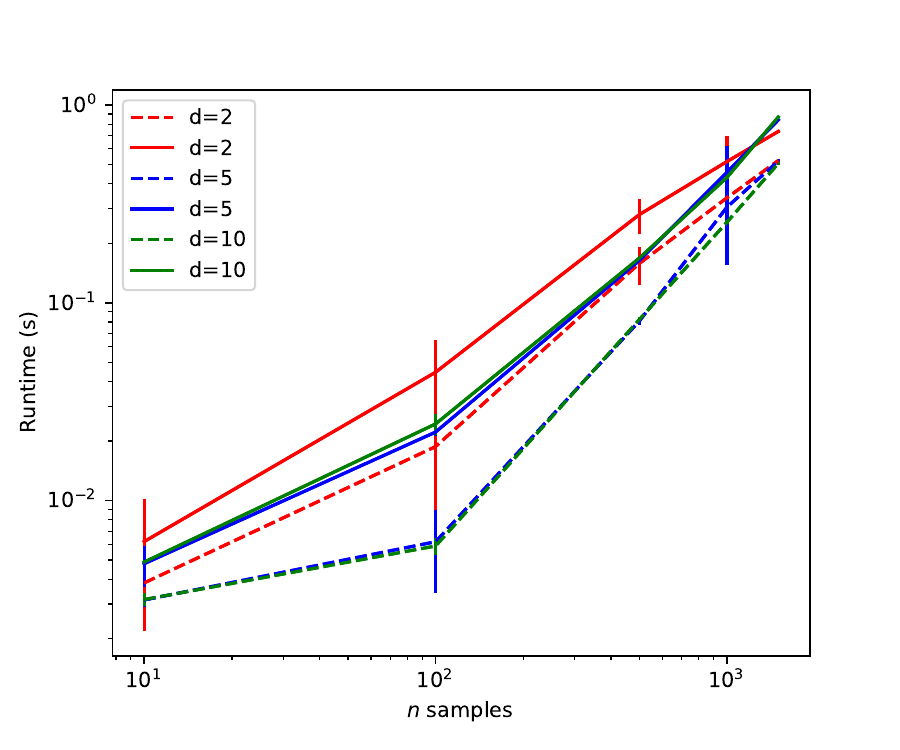}
            \caption{Runtime comparison}
    \end{subfigure}%
    \caption{Dashed lines are our estimator, solid lines are $\hat{T}^{\text{1NN}}$, and $T_0(x) = \exp(x)$}\label{fig: cpu_comp}
\end{figure}
We observe that in $d = 2$, the MSE of the two estimators are comparable, though our error deteriorates for large $n$, which reflects our slightly sub-optimal estimation rate. However, as $d$ increases to moderate dimensions, our estimator consistently outperforms $\hat{T}^{\text{1NN}}$ in both MSE and runtime with the choice of $\eps$ in \cref{thm:main}. For both estimators, the CPU runtime begins to become significant (on the order of seconds) when $n$ exceeds 1500.

\subsubsection{On estimating non-smooth transport maps}
We now consider the case of estimating non-smooth transport maps. Though we lack rigorous guarantees for this setting, our empirical findings suggest that our estimator nevertheless continues to perform well.
	
	Let $P = [-1,1]^d$ and let $\varphi_0(x) = 2|x_1| + \tfrac12 \|x\|^2$. This strongly convex function is differentiable $P$-almost everywhere, with gradient given by 
\begin{align*}
\nabla \varphi_0(x) = 2\text{sign}(x_1) + x\,.
\end{align*}
The resulting pushforward measure $(\nabla \varphi_0)_\sharp P$ has disconnected support, separated along the first coordinate. Mimicking the setup as before, we choose $\alpha=1$ for our choice of $\eps = \eps(n,\alpha)$ following the suggested parameters from \cref{thm:main}. Again, despite not fitting in our problem paradigm, the entropic map is able to out-perform the 1NN estimator in both runtime and MSE.
\begin{figure}[h]
\centering
    \begin{subfigure}[H]{0.5\textwidth}            
            \includegraphics[width=\textwidth]{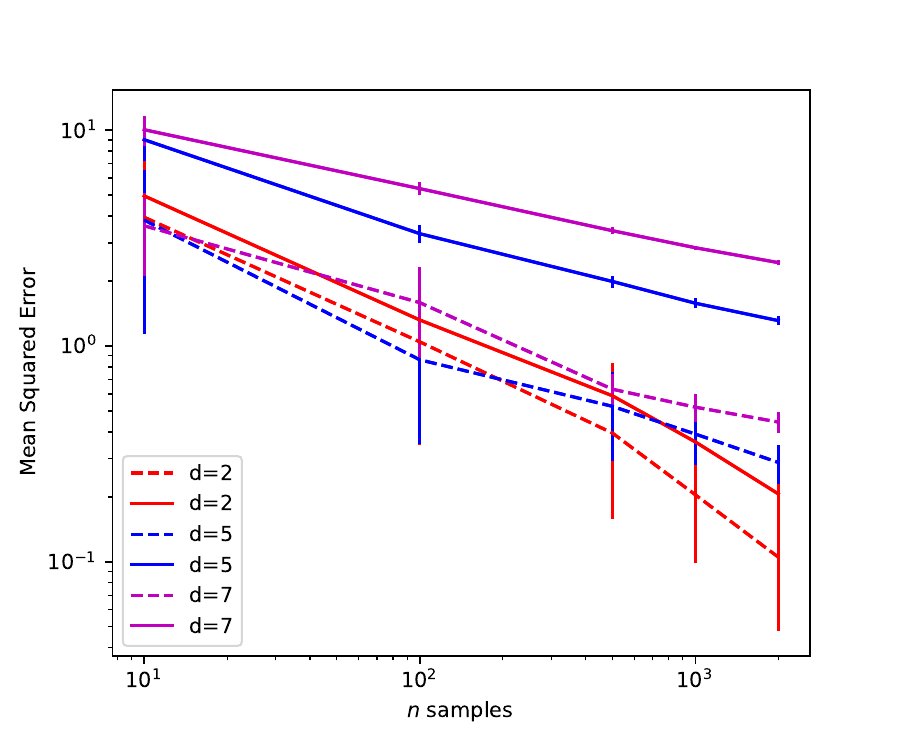}
            \caption{MSE comparison}
    \end{subfigure}%
    \begin{subfigure}[H]{0.5\textwidth}            
            \includegraphics[width=\textwidth]{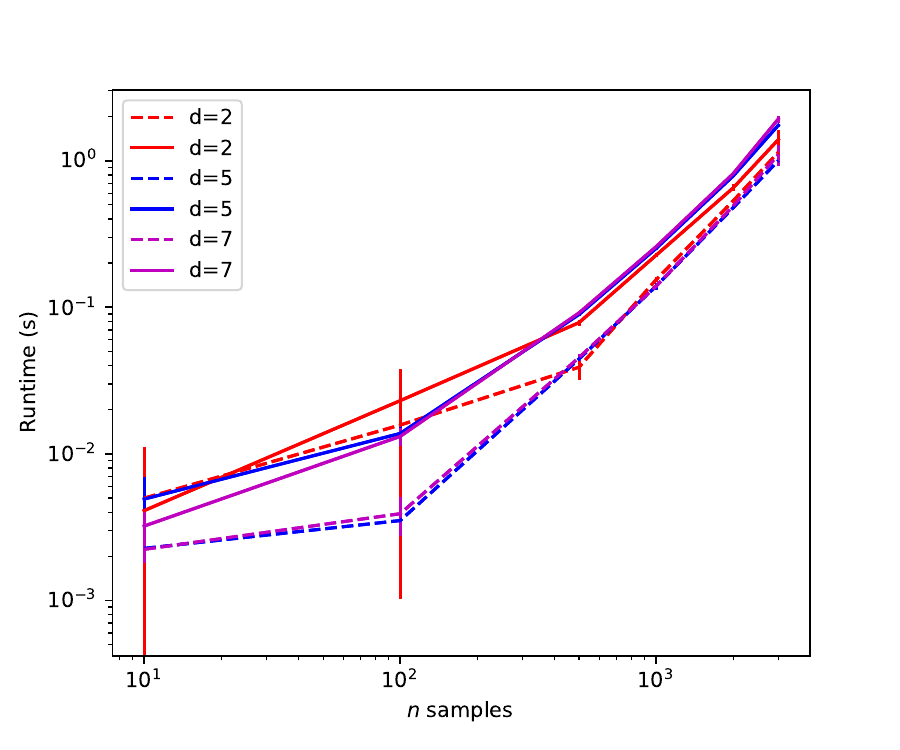}
            \caption{Runtime comparison}
    \end{subfigure}%
    \caption{Dashed lines are our estimator, solid lines are $\hat{T}^{\text{1NN}}$, and $T_0(x) = 2|x_1| + x$}\label{fig: cpu_comp_split}
\end{figure}

\subsubsection{Parallel estimation on massive data sets}
\Cref{fig: cpu_comp} makes clear that computation of both estimators slows for $n \gg 10^3$ when implemented on a CPU. However, Sinkhorn's algorithm can be easily parallelized. Unlike the 1-Nearest Neighbor estimator---and all other transport map estimators of which we are aware---our proposal therefore runs extremely efficiently on GPUs. We again average performance over 20 runs, and choose $\eps$ as in the previous example, with $T_0$ again as the exponential map (coordinate-wise). We see in \Cref{fig: gpu_results} that even when $n = 10^4$ and $d=10$, it takes roughly a third of a second to perform the optimization.
\begin{figure}[h]\label{fig: gpu_results}
\centering
    \begin{subfigure}[H]{0.5\textwidth}            
            \includegraphics[width=\textwidth]{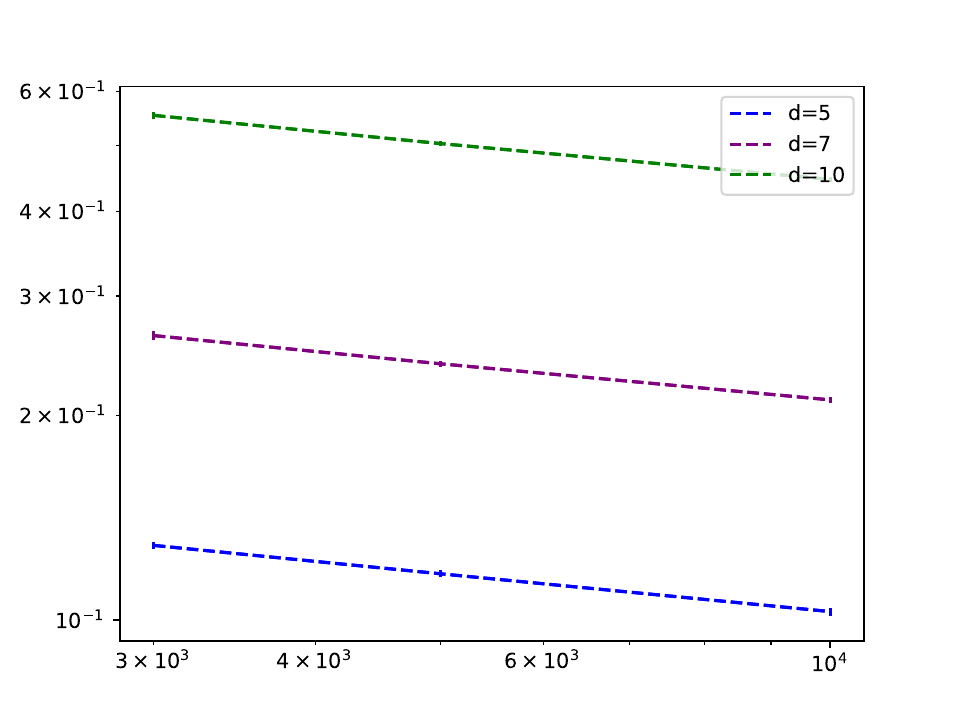}
            \caption{MSE comparison}
    \end{subfigure}%
    \begin{subfigure}[H]{0.5\textwidth}            
            \includegraphics[width=\textwidth]{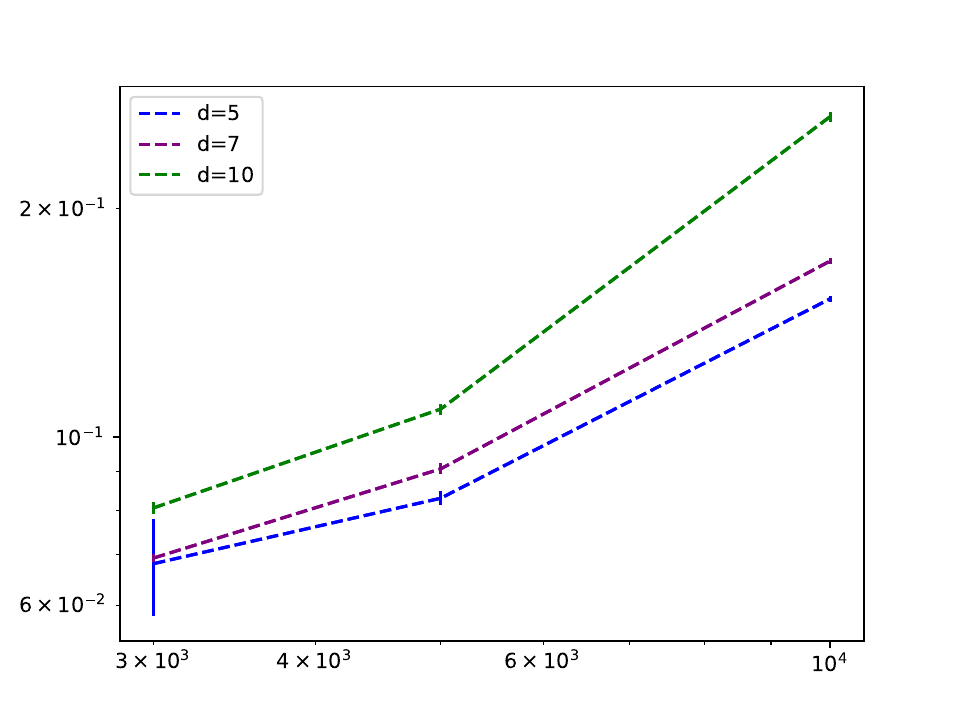}
            \caption{Runtime comparison}
    \end{subfigure}%
    \caption{Performance of a parallel implementation of our estimator on large data sets.}\label{fig: gpu_comp}
\end{figure}

\section{Conclusion and future directions}
We have presented the first finite-sample analysis of an entropic estimator for optimal transport maps.
The resulting estimator is easily parallelizable and fast to compute, even on massive data sets.
Though the theoretical rates we obtain fall short of minimax optimality, we demonstrate that our estimator empirically outperforms the other leading computationally tractable, statistically optimal proposal from the literature.

Based on the empirical success of our estimator, we conjecture that our analysis is loose, and that it may be possible to show that a properly tuned version of an entropic estimator achieves the minimax optimal rate, at least when $\alpha$ is small.
We also conjecture that our assumptions may be significantly loosened, and that similar results hold without stringent conditions on the densities or their support.

\subsubsection*{Acknowledgements}
AAP was supported in part by the Natural Sciences and Engineering Research Council of Canada, and the National Science Foundation under NSF Award 1922658 and grant DMS-2015291. JNW gratefully acknowledges the support of National Science Foundation grant
DMS-2015291. We thank Tudor Manole and Vincent Divol for helpful discussions, and the anonymous reviewers who greatly helped us to improve the quality of this paper. 

%%%%%%%%%%%%%%%%%%%%%%%%%%%%%%%%%%%%%%%%%%%%%%
%% Example with single Appendix:            %%
%%%%%%%%%%%%%%%%%%%%%%%%%%%%%%%%%%%%%%%%%%%%%%
%\begin{appendix}
%\section*{Title}\label{appn} %% if no title is needed, leave empty \section*{}.
%Appendices should be provided in \verb|{appendix}| environment,
%before Acknowledgements.
%
%If there is only one appendix,
%then please refer to it in text as \ldots\ in the \hyperref[appn]{Appendix}.
%\end{appendix}
%%%%%%%%%%%%%%%%%%%%%%%%%%%%%%%%%%%%%%%%%%%%%%
%% Example with multiple Appendixes:        %%
%%%%%%%%%%%%%%%%%%%%%%%%%%%%%%%%%%%%%%%%%%%%%%
\begin{appendix}

\section{Second-order error estimate}\label{sec: sec-order-est}
In this section, we outline a short proof of \cref{thm: faster_thm}.
%\begin{theorem} Suppose $P$ and $Q$ have bounded densities with compact support. Then
%\begin{align}\label{eq: tamanini_main}
%S_\eps(P,Q) - \frac{1}{2}W^2_2(P,Q)  + \eps \log((2 \pi \ep)^{d/2}) \leq - \frac{\eps}{2}\pran{\Ent(P) + \Ent(Q)} + \frac{\eps^2}{8}I_0(P,Q),
%\end{align}
%where $I_0(P,Q)$ is the integrated Fisher information along the Wasserstein geodesic between $P$ and $Q$.
%\end{theorem}
The proof hinges on the \textit{dynamic} formulations of $W_2^2(P,Q)$ and $S_\eps(P,Q)$ \citep{BenBre00, chizat2020faster,conforti2021formula}. We begin with the former:
\begin{align}\label{eq: dyn_w22}
\frac{1}{2}W_2^2(P,Q) = \inf_{\rho,v} \int_0^1 \int_{\R^d} \frac 12 \|v(t,x)\|^2_2 \rho(t,x) \dd x \dd t,
\end{align} 
subject to $\partial_t \rho + \nabla \cdot (\rho v) = 0$, called the \textit{continuity equation}, with $\rho(0,\cdot) = p(\cdot)$ and $\rho(1,\cdot) = q(\cdot)$. We let $(\rho_0,v_0)$ denote the joint minimizers to \cref{eq: dyn_w22} satisfying these conditions.

Similarly, there exists a dynamic formulation for $S_\eps$ \citep[see][for more information]{chizat2020faster,conforti2021formula}: for two measures with bounded densities and compact support,
\begin{align}\label{eq: dyn_se}
S_\eps(P,Q) + \eps\log(\Lambda_\eps) &= \inf_{\rho,v} \int_0^1 \int_{\R^d}  \pran{\frac 12 \|v(t,x)\|^2_2  + \frac{\eps^2}{8}\|\nabla_x \log(\rho(t,x))\|^2_2} \rho(t,x) \dd x \dd t \\
&- \frac{\eps}{2}(\Ent(P) + \Ent(Q) \nonumber,
\end{align}
subject to the same conditions as \cref{eq: dyn_w22}, where $\Lambda_\eps = (2\pi\eps)^{d/2}$.

If we plug in the minimizers from \cref{eq: dyn_w22} into \cref{eq: dyn_se}, we get exactly the result of \cref{eq: tamanini_main} by optimality
\begin{align*}
S_\eps(P,Q) + \eps\log(\Lambda_\eps) &\leq \int_0^1 \int_{\R^d} \frac{1}{2}\|v_0(t,x)\|^2_2 \rho_0(t,x)\dd x \dd t + \frac{\eps^2}{8}I_0(P,Q) - \frac{\eps}{2}(\Ent(P) + \Ent(Q)), \\
&= \frac 12 W_2^2(P,Q) + \frac{\eps^2}{8}I_0(P,Q) - \frac{\eps}{2}(\Ent(P) + \Ent(Q)),
\end{align*}
where we identify $I_0(P,Q) = \int_0^1 \int_{\R^d} \| \nabla_x \log \rho_0(t,x) \|_2^2 \rho_0(t,x) \dd x \dd t$.

\section{Laplace's method proof}\label{sec: laplace_method}
In this section, we prove a quantitative approximation to the integral
\begin{align}\label{eq: I_eps}
I(\eps) := \frac{1}{\Lambda_\eps}\int \exp\pran{-\frac{1}{\eps}f(x)}\dd x\,,
\end{align}
when $\eps \to 0$, with $f$ convex and sufficiently regular and where $\Lambda_\eps = (2\pi\eps)^{d/2}$.
This approximation relies on expanding $f$ around its global minimum; assuming that $f$ is twice-differentiable, the behavior of $f$ near its minimum will be quadratic, so that \cref{eq: I_eps} will resemble a Gaussian integral for $\eps$ sufficiently small.

Recall that for a positive definite matrix $S$, we define $J(S) := \sqrt{\det(S)}$.

In what follows, we write $\rd^2 f(0,y), \rd^3 f(0,y)$ for the second and third total derivative of $f$ at $x$, respectively. That is, for $y \in \R^d$ $$ \dd^2 f(x,y) := y^\top \nabla^2 f(x)y, \ \dd^3 f(x,y) := \sum_{i,j,k=1}^d \frac{\partial^3 f(x)}{\partial y_i \partial y_j \partial y_k}y_iy_jy_k\,. $$
We also define the set $B_r(a) := \{ y \in \R^d \ | \ \|y-a\| \leq r \} $, for some $r > 0$ and $a \in \R^d$. 
\begin{theorem}\label{thm: laplace_thm}
Let $I(\eps)$ be as in \cref{eq: I_eps}, with $f \in \cC^{\alpha+1}$, $m$-strongly convex, $M$-smooth, and $\alpha > 1$. Assume $f$ has a global minimum at $x^*$. Then there exist positive constants $c$ and $C$ depending on $m, M, \alpha, d, \text{and } \|f\|_{\cC^{\alpha + 1}}$ such that for all~$\eps \in (0, 1)$,
\begin{align}
c \leq J(\nabla^2f(x^*)) I(\eps) \leq 1 + C(\eps^{(\alpha-1)/2 \wedge 1})\,.
\end{align}
\end{theorem}
\begin{proof}
Without loss of generality, we may assume that $x^* = 0$.
For the remainder of the proof, we let $A := \nabla^2 f(0)$. Let $\tau = C_{m,M,d,\alpha}\sqrt{\log(2\eps^{-1})}$, where the constant is to be decided later. We split the desired integral into two parts:
\begin{align*}
I(\eps) &= \frac{1}{\Lambda_\eps} \int_{B_{\tau\sqrt{\ep}}(0)} e^{\frac{-1}{\eps}f(y)}\dd y + \frac{1}{\Lambda_\eps} \int_{B_{\tau\sqrt{\ep}}(0)^c} e^{\frac{-1}{\eps}f(y)}\dd y =: I_1(\eps) + I_2(\eps)\,.
\end{align*}
\paragraph*{Lower bounds}
Note that $I_2(\eps) \geq 0$, so it suffices to only prove $I_1(\eps) \geq \frac{c}{\sqrt{\det(A)}}$ for some constant $c > 0$.\\ \\
Since $f \in \cC^{\alpha+1}$, we have the following Taylor expansion
$$ -f(y) \geq -\frac{1}{2}y^\top A y - C\|y\|^{(\alpha+1) \wedge 3} \geq -\frac{M}{2}\|y\|^2 - C\|y\|^{(\alpha+1) \wedge 3}% =: -Q(y,0) - C\|y\|^{\alpha+1}\,,
$$
for some constant $C > 0$. Using this expansion, we arrive at
\begin{align*}
I_1(\eps) &= \frac{1}{\Lambda_\eps} \int_{B_{\tau\sqrt{\ep}}(0)}\exp\brac{-\frac{M}{2\eps}\|y\|^2 - \frac C \ep \|y\|^{(\alpha+1) \wedge 3}} \dd y \\
&\geq \frac{1}{\Lambda_\eps} \int_{B_{\tau\sqrt{\ep}}(0)}\exp\brac{-\frac{M}{2\eps}\|y\|^2 - \frac C \eps (\tau\sqrt{\ep})^{(\alpha+1) \wedge 3}} \dd y\,.
\end{align*}
Performing a change of measure and rearranging, we get
\begin{align*}
J(A)I_1(\eps) &\geq e^{-C(\tau\sqrt{\eps})^{(\alpha+1) \wedge 3}/\eps} \frac{J(A)}{(2M\pi)^{d/2}}\int_{B_{\tau\sqrt M(0)}}e^{-\frac{1}{2}\|y\|^2}\dd y\\
&\gtrsim e^{-C(\tau\sqrt{\eps})^{(\alpha+1) \wedge 3}/\eps} J(A) \mathbb{P}(\|Y\|\leq \tau\sqrt M)\,,
\end{align*}
where $Y \sim N(0,I_d)$.
Since $\alpha > 1$, the quantity $C(\tau\sqrt{\eps})^{(\alpha+1) \wedge 3}/\eps$ is bounded as $\eps \to 0$, so we may bound $e^{-C(\tau\sqrt{\eps})^{(\alpha+1) \wedge 3}/\eps}$ from below by a constant.
Since $J(A)$ and $\mathbb{P}(\|Y\|\leq \tau\sqrt M)$ are both also bounded from below, we obtain that $J(A) I_1(\eps) \geq c > 0$, as desired.

\paragraph*{Upper bounds}
We first show that the contribution from $I_2(\eps)$ is negligible. 
The strong convexity of $f$ implies $$f \geq \frac{m}{2}\|y\|^2,$$ leading us to the upper bound
\begin{align*}
I_2(\ep) & \leq \frac{1}{\Lambda_\ep}\int_{B_{\tau\sqrt{\ep}}(0)^c} e^{-\frac{m}{2\ep}\|y\|^2} \dd y \\
& =  \frac{1}{(2m\pi)^{d/2}}\int_{B_{\tau}(0)^c}e^{-\frac{1}{2}\|y\|^2}\dd y \\
& \leq \frac{1}{(2m\pi)^{d/2}}e^{-\frac 1 4 \tau^2}\int e^{-\frac{1}{4}\|y\|^2}\dd y \\
& \lesssim e^{-\frac{1}{4}\tau^2}\,,
\end{align*}
where in the penultimate inequality we have used the fact that $e^{-\frac 14 \|y\|^2} \leq e^{-\frac 14 \tau^2}$ on $B_\tau(0)^c$.
Taking $C_{m,M,d,\alpha}$ sufficiently large in the definition of $\tau$, we can make this term smaller than $\eps$.

For upper bounds on $I_1(\ep)$, we proceed in a similar fashion.
If $f \in \cC^{\alpha + 1}$ for $\alpha \in (1, 2]$, then we employ the bound
\begin{equation*}
- f(y) \leq - \frac{1}{2} y^\top A y + C \|y\|^{\alpha+1}\,,
\end{equation*}
yielding
\begin{align*}
I_1(\eps) &= \frac{1}{\Lambda_\ep}\int_{B_{\tau\sqrt{\ep}}(0)} e^{-\frac{1}{\eps}f(y)} \dd y \leq \frac{1}{\Lambda_\ep} \int_{B_{\tau\sqrt{\ep}}(0)} e^{-\frac{1}{2\eps}y^\top A y + \frac{C}{\eps}\|y\|^{\alpha+1}} \dd y \,.
\end{align*}
Performing the change of variables $u=\sqrt{1/\ep}y$, we arrive at
\begin{align*}
I_1(\ep) \leq \frac{1}{(2\pi)^{d/2}} \int_{B_\tau(0)}e^{\frac{-1}{2}u^\top A u}e^{C\ep^{(\alpha-1)/2}\|u\|^{\alpha+1}}\dd u
\end{align*}
Since $\alpha > 1$, the term $C\ep^{(\alpha-1)/2}\|u\|^{\alpha+1}$ is bounded above on $B_\tau(0)$, so that there exists a positive constant $C'$ such that
\begin{equation*}
e^{C\ep^{(\alpha-1)/2}\|u\|^{\alpha+1}} \leq 1+ C'\ep^{(\alpha-1)/2}\|u\|^{\alpha+1} \quad \forall u \in B_\tau(0)\,.
\end{equation*}
We obtain
\begin{align*}
I_1(\ep) & \leq \frac{1}{(2\pi)^{d/2}}\int_{B_\tau(0)}e^{\frac{-1}{2} u^\top A u}(1 + C'\ep^{(\alpha-1)/2}\|u\|^{\alpha+1}) \dd u \\
& \leq \frac{1}{(2\pi)^{d/2}}\int e^{\frac{-1}{2} u^\top A u}(1 + C'\ep^{(\alpha-1)/2}\|u\|^{\alpha+1}) \dd u.
\end{align*}
Performing another change of variables yields
\begin{equation*}
I_1(\ep) \leq \frac{1}{(2 \pi)^{d/2} J(A)} \int (1 + C' \ep^{(\alpha - 1)/2} \|A^{-1/2} u\|^{\alpha + 1}) e^{- \frac 12 \|u\|^2} \dd u
\end{equation*}
We obtain
\begin{equation*}
J(A) I_1(\ep) \leq 1 + C'' \ep^{(\alpha - 1)/2}\,.
\end{equation*}
Combining this with the bound on $J(A) I_2(\ep)$ yields the bound for $\alpha \leq 2$.

When $\alpha > 2$, we use the same technique but expand to the third order, yielding
\begin{align*}
I_1(\eps) & = \frac{1}{\Lambda_\ep}\int_{B_{\tau\sqrt{\ep}}(0)} e^{-\frac{1}{\eps}f(y)} \dd y \\
& \leq \frac{1}{\Lambda_\ep} \int_{B_{\tau\sqrt{\ep}}(0)} e^{-\frac{1}{2\eps}y^\top A y - \frac{1}{6\eps}\dd^3f(0,y) + \frac{C}{\eps}\|y\|^{\alpha+1}} \dd y \\
& = \frac{1}{(2\pi)^{d/2}} \int_{B_{\tau}(0)} e^{-\frac{1}{2}u^\top A u - \frac{\eps^{1/2}}{6}\dd^3f(0,u) + C\eps^{(\alpha-1)/2}\|u\|^{\alpha+1}} \dd u
\end{align*}
Since $- \frac{\eps^{1/2}}{6}\dd^3f(0,u) + C\eps^{(\alpha-1)/2}\|u\|^{\alpha+1}$ is bounded on $B_\tau(0)$, we have
\begin{equation*}
e^{- \frac{\eps^{1/2}}{6}\dd^3f(0,u) + C\eps^{(\alpha-1)/2}\|u\|^{\alpha+1}} \leq 1 - \frac{\eps^{1/2}}{6}\dd^3f(0,u) + C\eps^{(\alpha-1)/2}\|u\|^{\alpha+1} + R(u)\,,
\end{equation*}
where $R$ is a positive remainder term satisfying $R(u) \lesssim \eps (\dd^3f(0,u))^2 + \eps^{\alpha - 1}\|u\|^{2(\alpha+1)}$.
We obtain
\begin{equation*}
I_1(\eps)  \leq \frac{1}{(2\pi)^{d/2}} \int_{B_\tau(0)} \left(1 - \frac{\eps^{1/2}}{6}\dd^3f(0,u) + C\eps^{(\alpha-1)/2}\|u\|^{\alpha+1} + R(u)\right) e^{-\frac 12 u^\top A u} \dd u\,.
\end{equation*}
The symmetry of $B_\tau(0)$ and the fact that $\dd^3f(0,u)e^{-\frac 12 u^\top A u}$ is an odd function of $u$ imply
\begin{equation*}
\int_{B_\tau(0)} \dd^3f(0,u) e^{-\frac 12 u^\top A u} \dd u = 0\,,
\end{equation*}
so
\begin{align*}
I_1(\eps)  & \leq\frac{1}{(2\pi)^{d/2}} \int (1 + C\eps^{(\alpha-1)/2}\|u\|^{\alpha+1} + R(u)) e^{-\frac 12 u^\top A u} \dd u \\
& = \frac{1}{(2\pi)^{d/2}J(A)} \int (1 + C\eps^{(\alpha-1)/2}\|A^{-1/2}u\|^{\alpha+1} + R(A^{-1/2}u))e^{-\frac 12 \|u\|^2} \dd u \\
& \leq 1 + C'' \eps^{(\alpha -1)/2} + C'' \eps\,,
\end{align*}
which is the desired bound.
\end{proof}

\begin{corollary}\label{laplace_z}
Assume \textbf{(A2)} and  \textbf{(A3)}. For all $\alpha \in (1, 3]$, there exist positive constants $c$ and $C$ such that
\begin{align}
c \leq J(\nabla^2\varphi_0^*(x^*))Z_\eps(x)\leq 1 + C \eps^{(\alpha-1)/2}\,,
\end{align}
for all $\eps \in (0, 1)$ and $x \in \supp(P)$.
\end{corollary}
\begin{proof}
Take $f(\cdot) = D[\cdot|x^*]$ which is $1/L$-strongly convex, and $1/\mu$-smooth, with minimizer $x^*$ (see \cref{quad_bounds}).
The claim now follows from \cref{thm: laplace_thm}.
\end{proof}

\section{Omitted proofs}\label{sec: omitted_proofs}
\subsection{Proof of \cref{prop:new_dual}}
It suffices to show that
\begin{equation*}
S_\eps(P, Q) \geq \sup_{\eta \in L^1(\pi_\ep)} \int \eta \dd \pi_\ep - \ep \iint e^{(\eta(x, y) - \frac 12 \|x - y\|^2)/\ep} \dd P(x) \dd Q(y)  + \ep\,,
\end{equation*}
since the other direction follows from choosing $\eta(x, y) = f(x) + g(y)$ and using \cref{eq: dual_eot}.

Write
\begin{equation*}
\gamma(x, y) = e^{\frac 1 \eps (f_\eps (x) + g_\eps (y) - \frac 12 \|x - y\|^2)}
\end{equation*}
for the $P \otimes Q$ density of $\pi_\eps$.
The inequality
\begin{equation*}
a \log a \geq ab - e^b + a
\end{equation*}
holds for all $a \geq 0$ and $b \in \RR$, as can be seen by noting that the right side is a concave function of $b$ which achieves its maximum at $b = \log a$.
Applying this inequality with $a = \gamma(x, y)$ and $b = \frac 1 \eps (\eta(x, y) - \frac 12 \|x - y\|^2)$ and integrating with respect to $P \otimes Q$ yields
\begin{align*}
\int \log \gamma \dd \pi_\eps \geq \frac 1 \eps \bigg(\int \eta \dd \pi_\eps &-  \int \frac 12 \|x - y\|^2 \dd \pi_\eps(x, y)\bigg) \\
&- \iint e^{(\eta(x, y) - \frac 12 \|x - y\|^2)/\eps} \dd P(x) \dd Q(y) + 1
\end{align*}
Multiplying by $\eps$ and using the fact that
\begin{align*}
\int \eps \log \gamma \dd \pi_\eps & = \int (f_\eps(x) + g_\eps(y) - \frac 12 \|x - y\|^2) \dd \pi_\eps = S_\eps(P, Q) - \int \frac 12 \|x - y\|^2 \dd \pi_\eps(x, y)
\end{align*}
yields the claim.
\qed

\subsection{Proof of Proposition~\ref{prop:two_sample_main}}
\Cref{prop:two_sample_main} follows from the following more general result, which recovers \cref{prop:two_sample_main} by choosing $\tilde P = P_n$.

\begin{proposition}\label{prop: delta_dev}
Let $P$ and $Q$ be probability measures with support contained in $\Omega$, and denote by $P_n$ and $Q_n$ corresponding empirical measures. If $\tilde{P}$ is a probability measure with support in $\Omega$ such that $\tv{\tilde{P}}{P_n} \leq \delta$ for some $\delta\geq 0$, then 
\begin{align*}
\E \Big\{\sup_{\chi: \Omega \times \Omega \to \RR} \iint \chi(x, y) \dd \pi_{\eps, n}(x, y) &- \iint (e^{\chi(x, y)} - 1) \tilde{\gamma}(x, y) \dd P(x) \dd Q_n(y)\Big\} \\
& \lesssim \eps^{-1}\delta + (\eps^{-1} + \eps^{-d/2}) \log(n) n^{-1/2}\,,
\end{align*}
where $\pi_{\ep, n}$ is the optimal entropic plan for $P$ and $Q_n$, $\tilde{\gamma}$ is the $\tilde P \otimes Q_n$ density of the optimal entropic plan for $\tilde P$ and $Q_n$, and the supremum is taken over all $\chi \in L^1(\pi_{\ep, n}).$
\end{proposition}
\begin{proof}
	Write $\tilde f_\eps$ and $\tilde g_\eps$ for the optimal entropic potentials for $\tilde P$ and $Q_n$, so that $$\tilde{\gamma}(x, y) = \exp\pran{\eps^{-1}(\tilde{f}_\eps(x) + \tilde{g}_\eps(y) - \frac 12 \|x-y\|^2)}\,.$$
Plugging in $\eta(x, y) = \eps \chi(x, y) + \tilde{f}_\eps(x) + \tilde{g}_\eps(y)$ into \cref{prop:new_dual} gives
\begin{align*}
\sup_{\chi: \Omega \times \Omega \to \RR} \iint \chi \dd \pi_{\eps, n} - \iint (e^{\chi(x, y)} -1) \tilde{\gamma}(x, y) \dd P(x) \dd Q_n(y) &\leq \eps^{-1} \bigg(S_\eps(P, Q_n) \\
&- \int  \tilde{f}_\eps \dd P - \int \tilde{g}_\eps \dd Q_n\bigg)\,,
\end{align*}
where we have used that $\tilde \gamma$ is a probability density with respect to $P \otimes Q_n$ by \cref{dual_opt}.

Let $f_{\eps, n}$ and $g_{\eps, n}$ be the optimal entropic dual potentials for $P$ and $Q_n$. 
As in the proof of \cref{empirical_process}, the optimality of $\tilde{f}_\eps$ and $\tilde{g}_{\eps}$ for the pair $(\tilde{P}, Q_n)$ implies
\begin{align*}
\int\tilde{f}_\eps \dd \tilde{P} + \int \tilde{g}_{\eps} \dd Q_n & \geq \int f_{\eps, n} \dd \tilde{P} + \int g_{\eps, n} \dd Q_n \\
&- \eps \iint e^{\frac 1 \eps( f_{\eps, n}(x) + g_{\eps, n}(y) - \frac 12 \|x - y\|^2)} \dd \tilde{P}(x) \dd Q_n(y) + \eps \\
& = \int f_{\eps, n} \dd \tilde{P} + \int g_{\eps, n} \dd Q_n\,,
\end{align*}
since $\int e^{\frac 1 \eps( f_{\eps, n}(x) + g_{\eps, n}(y)- \frac 12 \|x - y\|^2)} \dd Q_n(y) = 1$ by the dual optimality condition in \cref{dual_opt}.
Therefore
\begin{align*}
S_\eps(P, Q_n) - \int  \tilde{f}_\eps \dd P - \int \tilde{g}_\eps \dd Q_n &\leq \int (f_{\eps, n} - \tilde{f}_\eps) (\rd P - \rd \tilde{P})\,\\
&= \int (f_{\eps, n} - \tilde{f}_\eps) (\rd P - \rd {P}_n) + \int (f_{\eps, n} - \tilde{f}_\eps) (\rd P_n - \rd \tilde{P})
\end{align*}
By \citet[Proposition 1]{GenChiBac18}, we may choose $f_{\eps, n}$ and $\tilde f_\eps$ to satisfy $\norm{f_{\eps, n}}_\infty, \|\tilde f_\eps\|_\infty \lesssim 1$, so we may bound the second term as
\begin{equation*}
 \int (f_{\eps, n} - \tilde{f}_\eps) (\rd P_n - \rd \tilde{P}) \lesssim \tv{\tilde{P}}{P_n} \leq \delta\,.
\end{equation*}
Also, since $f_{\eps, n}$ is independent of $P_n$,
\begin{equation*}
\E f_{\eps, n}(\rd P - \rd P_n)(y) = 0\,.
\end{equation*}
Altogether, we obtain 
\begin{multline*}
\E \sup_{\chi: \Omega \times \Omega \to \RR} \iint \chi \dd \pi_{\eps, n} - \iint (e^{\chi(x, y)} -1) \gamma(x, y) \dd P(x) \dd Q_n(y)\\ \lesssim \eps^{-1}\bigg( \delta 
+ \E \int \tilde{f}_\eps (\rd P_n - \rd P)\bigg)\,.
\end{multline*}
We conclude by again appealing to \citet[Proposition 1]{GenChiBac18}: since $\tilde{f}_\eps$ is an optimal entropic potential for the pair of compactly distributed probability measures $(\tilde{P},Q_n)$, its derivatives up to order $s$ are bounded by $C_{s, d, K}(1 + \eps^{1-s})$ on any compact set $K$ for any $s \geq 0$.
Taking $K$ to be a suitably large ball containing $\Omega$ and applying \Cref{empirical_process} with $s = d/2$ yields the claim.
\end{proof}

\subsection{Proofs from \cref{sec: one_samp}}

\begin{proof}[Proof of \cref{lem:subg}]
Fix $x \in \supp(P)$ and let $x^* := T_0(x)$, and for notational convenience, write $Y$ for the random variable with density $q_\eps^x$, and denote its mean by $\bar{Y}$.
It suffices to show the existence of a constant $K$ such that for any unit vector $v$, 
\begin{equation}\label{eq: subg1}
\E e^{(v^\top(Y-x^*))^2/4 L \eps} \leq K\,.
\end{equation}
Indeed, by Young's and Jensen's inequalities, this implies
\begin{equation*}
\E e^{(v^\top(Y-\bar Y))^2/8 L \eps} \leq e^{(v^\top(\bar Y-x^*))^2/4 L \eps} \E e^{(v^\top(Y-x^*))^2/4 L \eps} \leq K^2\,,
\end{equation*}
and hence by another application of Jensen's inequality that
\begin{equation*}
\E e^{(v^\top(Y-\bar Y))^2/C \eps} \leq 2
\end{equation*}
for $C = 8 L K^2$.

We prove \cref{eq: subg1} using the strong convexity of $D[y | x^*]$.
By \cref{quad_bounds},
\begin{align*}
\E e^{(v^\top(Y-x^*))^2/4 L \eps} &\leq \frac{1}{Z_{\eps}(x) \Lambda_\eps} \int e^{- \frac 1 \eps D[y | x^*] + \frac{1}{4 L \eps} \|y - x^*\|^2} \dd y \\
& \leq \frac{1}{Z_{\eps}(x) \Lambda_\eps} \int e^{- \frac{1}{4 L \eps} \|y - x^*\|^2} \dd y \\
& = \frac{(2L)^{d/2}}{Z_\eps(x)} \\
& \lesssim 1\,,
\end{align*}
where the final inequality uses \cref{laplace_z}.

\end{proof}

\begin{proof}[Proof of \cref{lem:exp_bound}]
Let us first fix an $x \in \supp(P)$, and write $Y$ for the random variable with density $q_\eps^x$ and $\bar Y$ for its mean, and write $x^* := T_0(x)$. \Cref{lem:subg} implies \citep[see][Proposition 2.5.2]{Ver18} that there exists a positive constant $C$, independent of $x$, such for any $v \in \RR^d$, 
\begin{equation*}
\E e^{(v^\top(Y - x^*))} = e^{v^\top(\bar Y - x^*)} \E e^{(v^\top(Y - \bar Y))}  \leq e^{ v^\top(\bar Y - x^*) + C \eps \|v\|^2} \leq e^{ \frac 1{4\eps} \|\bar Y - x^*\|^2 + (C+ 1) \eps \|v\|^2}\,,
\end{equation*}
where the last step uses Young's inequality.
Equivalently, for $a > (C+1)\eps$, we have for all $x \in \supp(P)$ and $v \in \RR^d$
\begin{equation*}
\int e^{(v^\top(y - x^*)) - a \|v\|^2} q_\eps^x(y) \dd y \leq e^{ \frac 1{4\eps} \|\bar y^x -x^*\|^2 } \,.
\end{equation*}
Applying this inequality with $v = h(x)$ and integrating with respect to $P$ yields the claim.
\end{proof}

\begin{proof}[Proof of \cref{lem:mean}]
It suffices to prove the claim for $\alpha \in (1, 2]$.
Let us fix an $x \in \supp(P)$. Since $\varphi_0^* \in \cC^{\alpha+1}(\Omega)$, Taylor's theorem implies
\begin{equation*}
D[y|x^*] = - x^\top y + \varphi_0(x) + \varphi_0^*(y) = \frac 12 (y- x^*)^\top \nabla^2 \varphi_0^*(x^*) (y- x^*) + R(y|x^*)\,,
\end{equation*}
where the remainder satisfies
\begin{equation}\label{eq:taylor_rem}
|R(y|x^*)| \lesssim \|y - x^*\|^{1+\alpha}\,.
\end{equation}

We aim to bound
\begin{align*}
\left\|\bar{y}^x - x^*\right\| & = \left\|\frac{1}{Z_\ep(x) \Lambda_\ep} \int (y - x^*) e^{- \frac{1}{\ep}D[y|x^*]} \dd y \right\|
%\frac{1}{Z_\ep(x) \Lambda_\ep} \int (y - x^*) e^{- \frac 1 \ep R(y|x^*)} e^{- \frac{1}{2\ep}(y- x^*)^\top A(x^*) (y- x^*)}  \dd y \\
%& = \frac{1}{Z_\ep(x) \Lambda_\ep} \int (y - x^*) \left(e^{- \frac 1 \ep R(y|x^*)} - 1\right) e^{- \frac{1}{2\ep}(y- x^*)^\top A(x^*) (y- x^*)}  \dd y\,,
\end{align*}

Let 
$\tau = C \sqrt{\log(2 \eps^{-1})}$ for a sufficiently large constant $C$.
As in the proof of \cref{thm: laplace_thm}, the contribution to the integral from the set $B_{\tau \sqrt \eps}(x^*)^c$ is negligible; indeed, \cref{quad_bounds} implies
\begin{align*}
& \left\|\frac{1}{Z_\ep(x) \Lambda_\ep} \int_{B_{\tau \sqrt \eps}(x^*)^c} (y - x^*) e^{- \frac{1}{\ep}D[y|x^*]} \dd y\right\| \\
&\qquad \leq \frac{1}{Z_\ep(x) \Lambda_\ep} \int_{B_{\tau \sqrt \eps}(x^*)^c} \|y - x^*\| e^{- \frac{1}{2 L\ep}\|y - x^*\|^2} \dd y \\
 &\qquad  = \frac{\ep^{(d+1)/2}}{Z_\ep(x) \Lambda_\ep} \int_{B_{\tau}(0)^c} \|y\| e^{- \frac{1}{2 L} \|y\|^2} \dd y \\
 &\qquad  \leq \frac{\ep^{(d+1)/2}}{Z_\ep(x) \Lambda_\ep} \left(\int \|y\|^2 e^{- \frac{1}{2 L} \|y\|^2} \dd y\right)^{1/2} \left(\int_{B_{\tau}(0)^c} e^{- \frac{1}{2 L} \|y\|^2} \dd y\right)^{1/2} \\
&\qquad \lesssim \eps^{1/2} \mathbb{P}[\|Y\| \geq \tau]\,, \quad \quad Y \sim \cN(0, I_d)\,,
\end{align*}
and this quantity can be made smaller than $\eps$ by choosing the constant in the definition of $\tau$ sufficiently large.

It remains to bound
\begin{multline}\label{eq:near_laplace}
\left\|\frac{1}{Z_\ep(x) \Lambda_\ep} \int_{B_{\tau \sqrt \eps}(x^*)} (y - x^*) e^{- \frac 1 \ep R(y|x^*)} e^{- \frac{1}{2\ep}(y- x^*)^\top \nabla^2 \varphi_0^*(x^*) (y- x^*)}  \dd y\right\|  = \\
\left\|\frac{1}{Z_\ep(x) \Lambda_\ep} \int_{B_{\tau \sqrt \eps}(x^*)} (y - x^*) \left(e^{- \frac 1 \ep R(y|x^*)} - 1\right) e^{- \frac{1}{2\ep}(y- x^*)^\top \nabla^2 \varphi_0^*(x^*) (y- x^*)}  \dd y\right\|\,,
\end{multline}
where we have used that
\begin{equation*}
\int_{B_{\tau \sqrt \eps}(x^*)} (y - x^*)e^{- \frac{1}{2\ep}(y- x^*)^\top \nabla^2 \varphi_0^*(x^*) (y- x^*)}  \dd y = 0\,.
\end{equation*}

By \eqref{eq:taylor_rem},
\begin{equation*}
\frac 1 \ep |R(y|x^*)| \lesssim \frac{1}{\ep} \|y - x^*\|^{1+\alpha} \lesssim 1 \quad \quad \forall y \in B_{\tau \sqrt \eps}(x^*)\,,
\end{equation*}
and since $|e^t - 1| \lesssim |t|$ for $|t| \lesssim 1$, we obtain that
\begin{equation*}
\left|e^{- \frac 1 \ep R(y|x^*)} - 1\right| \lesssim \frac{1}{\ep} \|y - x^*\|^{1+\alpha} \quad \quad \forall y \in B_{\tau \sqrt \eps}(x^*)\,.
\end{equation*}
Therefore~\eqref{eq:near_laplace} is bounded above by
\begin{align*}
%\frac{1}{Z_\ep(x) \Lambda_\ep} \int_{B_{\tau \sqrt \eps}(x^*)} \frac{M}{\eps} \|y - x^*\|^{2+\alpha} e^{- \frac{1}{2\ep}(y- x^*)^\top A(x^*) (y- x^*)}  \dd y \leq 
\frac{C}{Z_\ep(x) \Lambda_\ep \ep} \int_{\RR^d} \|y - x^*\|^{2+\alpha} e^{- \frac{1}{2\ep}(y- x^*)^\top \nabla^2 \varphi_0^*(x^*) (y- x^*)} \dd y \lesssim\ep^{\alpha/2}\,,
\end{align*}
where in the last step we have applied \cref{laplace_z}.
We therefore obtain
\begin{equation*}
\|\E_{q_\eps^x}(Y) - x^*\|\  \lesssim \ep^{\alpha/2} + \ep\,.
\end{equation*}
Taking squares, we get the desired result.
\end{proof}

\section{Supplementary results}\label{sec: remaining_proofs}

\begin{proposition}\label{exp_empirical_process}
For any $x \in \supp(P)$, if $a \in [L \eps, 1]$, then
\begin{equation*}
\E \sup_{h: \Omega \to \RR^d} \int e^{j_h(x, y)} \frac{q_\eps^x(y)}{q(y)} (\rd Q_n - \rd Q)(y)\lesssim (1+\ep^{-d/2})  n^{-1/2}\,,
\end{equation*}

where the implicit constant is uniform in $x$.
\end{proposition}
\begin{proof}

To bound this process, we employ the following two lemmas:
\begin{lemma}\label{quadratic_ub}
If $a \geq L\ep$, then for any $v \in \RR^d$,
\begin{equation*}
v^\top(y -x^*) - a \|v\|^2 - \frac 1 \ep D[y | x^*] \leq - \frac{\ep L}{2} \|v\|^2\,.
\end{equation*}
\end{lemma}
\begin{proof}
By \cref{quad_bounds}, $D[y | x^*] \geq \frac{1}{2 L} \|y - x^*\|^2$.
Combining this fact with Young's inequality yields
\begin{align*}
v^\top(y -x^*) - a \|v\|^2 - \frac 1 \ep D[y | x^*] & \leq \frac{\ep L}{2} \|v\|^2 + \frac{1}{2 \ep L} \|y - x^*\|^2 - a \|v\|^2 - \frac 1 \ep D[y | x^*] \\
& \leq - \frac{\ep L}{2} \|v\|^2\,, 
\end{align*}
as claimed.
\end{proof}
By slight abuse of notation, for any $v \in \RR^d$, write $j_v: \Omega \to \RR$ for the function
\begin{equation*}
j_v(y) = v^\top (y - T_0(x)) - a \|v\|^2\,.
\end{equation*}
Let
\begin{align}\label{eq: J_eps_class}
\cJ_\eps = \{e^{j_v} \frac{q_\eps^x(y)}{q(y)}: v \in \RR^d\}
\end{align}
\begin{lemma}\label{lem: J_eps_class}
If $a \in [L \ep, 1]$, then
\begin{equation*}
\log N(\tau, \cJ_\eps, \|\cdot\|_{L^\infty(Q)}) \lesssim d \log(K/\tau)\,,
\end{equation*}
where $K \lesssim (1+\ep^{-d/2})$.
\end{lemma}
\begin{proof}
	Fix $\delta \in (0, 1)$.
Let $\cN_\delta$ be a $\delta^{3/2}$-net with respect to the Euclidean metric of a ball of radius $\delta^{-1/2}$ in $\RR^d$, and consider the set
\begin{equation*}
\cG_{\delta} \defeq \left\{e^{j_v} \frac{q_\eps^x(y)}{q(y)}: v \in \cN_\delta\right\} \cup \{e^{j_w} \frac{q_\eps^x(y)}{q(y)}\}\,,
\end{equation*}
where $w \in \RR^d$ is an arbitrary vector of norm $ \delta^{-1/2}$.
By \cref{quadratic_ub}, if $a > L \ep$ and $\|v\| \geq R$, then
\begin{equation*}
\sup_{y \in \supp(Q)} e^{j_v} \frac{q_\eps^x(y)}{q(y)} \leq \sup_{y \in \supp(Q)} \frac{1}{Z_\eps(x) \Lambda_\ep q(y)} e^{- \frac{1}{2\ep L} R^2} \leq \sup_{y \in \supp(Q)} \frac{2 L}{Z_\eps(x) \Lambda_\ep q(y) R^2}\,.
\end{equation*}
Therefore, if $v \in \RR^d$ satisfies $\|v\| \geq \delta^{-1/2}$, then
\begin{equation*}
\sup_{y \in \supp(Q)} \left| e^{j_w(y)}\frac{q_\eps^x(y)}{q(y)} - e^{j_v(y)}\frac{q_\eps^x(y)}{q(y)}\right| \leq \sup_{y \in \supp(Q)} \frac{4 \delta L}{Z_\eps(x) \Lambda_\ep q(y)} \leq K \delta
\end{equation*}
for $K = \sup_{y \in \supp(Q)} \frac{1+4 L}{Z_\ep(x) \Lambda_\ep q(y)} \lesssim \eps^{-d/2}$.

On the other hand, if $v \in \RR^d$ satisfies $\|v\| \leq  \delta^{-1/2}$, pick $u \in \cN_\delta$ satisfying $\|u - v\| \leq \delta^{3/2}$.
We then have
\begin{equation*}
\left| e^{j_u(y)}\frac{q_\eps^x(y)}{q(y)} - e^{j_v(y)}\frac{q_\eps^x(y)}{q(y)}\right| \leq \frac{|j_u(y) - j_v(y)|}{q(y)}\,,
\end{equation*}
where we have used \cref{quadratic_ub} combined with the inequality
\begin{equation*}
|e^a - e^b| \leq |a - b| \quad \quad \forall a, b \leq 0\,.
\end{equation*}

Since $\|u - v\| \leq \delta^{3/2}$ and $\|u\| + \|v\| \leq 2 \delta^{-1/2}$, we have for any $y \in \Omega$,
\begin{equation*}
|j_u(y) - j_v(y)| = |(u-v)^\top(y -T_0(x)) - a (\|u\|^2 - \|v\|^2)| \lesssim \delta^{3/2} + \delta a\,,
\end{equation*}
where we have used the fact that $y$ and $T_0(x)$ lie in the compact set $\Omega$.
Therefore, as long as $a \leq 1$, this quantity is bounded by $C \delta$ for a positive constant $C$.

All told, we obtain that for any $v \in \RR^d$, there exists a $g \in \cG_\delta$ such that
\begin{equation*}
\left\|e^{j_v}\frac{q_\eps^x(y)}{q(y)} - g\right\|_{L^\infty(Q)} \lesssim K \delta\,,
\end{equation*}
where $K \lesssim 1 + \eps^{-d/2}$.
Moreover, \cref{quadratic_ub} implies that, for any $g \in \cG_\delta$,
\begin{equation*}
\|g\|_{L^\infty(Q)} \leq \sup_{y \in \supp(Q)} \frac{1}{Z_\eps(x) \Lambda_\eps q(y)} \leq K.
\end{equation*}

By a volume argument, we may choose $\cN_\delta$ such that it satisfies
\begin{equation*}
\log |\cN_\delta| \lesssim \log(1/\delta).
\end{equation*}
We therefore obtain for any $\tau \leq K$,
\begin{equation*}
\log N(\tau, \cJ_\eps, \|\cdot\|_{L^\infty(Q)}) \leq \log |\cG_{\tau/K}| \lesssim \log(K/\tau)\,,
\end{equation*}
as claimed.
\end{proof}

Returning to the empirical process, we obtain by a chaining bound~\citep[Theorem 3.5.1]{GinNic16}
\begin{align*}
\E \sup_{h: \Omega \to \RR^d} \int e^{j_h(x, y)} \frac{q_\eps^x(y)}{q(y)} (\rd Q_n - \rd Q)(y) & = \E \sup_{j \in \cJ} \int j(y) (\rd Q_n - \rd Q)(y) \\
&\lesssim n^{-1/2} \int_0^K \sqrt{\log(K/ \tau)} \dd \tau \\
& \lesssim K n^{-1/2}\,.
\end{align*}
Recalling that $K \lesssim (1+ \eps^{-d/2})$ completes the proof.
\end{proof}

\begin{lemma}\label{empirical_process}
For a convex, compact $K \sse \R^d$, for any real number $s \geq d/2$, and $M > 0$, let $\cC^s(K;M)$ be the set of $s$-H\"older smooth functions with H\"older norm bounded by $M$. For any probability measure $\nu$ with support contained in $K$ and corresponding empirical measure $\nu_n$, we have that
\begin{equation*}
\E \sup_{g \in \cC^s(K;M)} \int  g(y) (\rd \nu_n(y) - \rd \nu(y)) \lesssim C_K M \log(n) n^{-1/2}\,.
\end{equation*}
\end{lemma}

\begin{proof}
We write $\cF$ to be the set of functions in $\cC^s(K;1)$. A version of Dudley's chaining bound \cite[see, e.g.,][Theorem 16]{LuxBou04} therefore implies for any $\delta \geq 0$,
\begin{align*}
\E \sup_{g \in \cC^s(K;M)} \int  g(y) (\rd \nu_n(y) - \rd \nu(y)) 
&\lesssim M \left(\delta + n^{-1/2}\int_\delta^1 \sqrt{\log N(\tau, \cF, \|\cdot\|_\infty)} \dd \tau\right)\,.
\end{align*}

Letting $s \geq d/2$ and applying standard covering number bounds for H\"older spaces \citep[Theorem 2.7.1]{VaaWel96} implies
\begin{equation*}
\E\sup_{g \in \cC^s(K;M)} \int   g(y) (\rd \nu_n(y) - \rd \nu(y)) \lesssim C_K \inf_{\delta \geq 0} M \left(\delta + n^{-1/2}\int_\delta^1 \tau^{-1} \dd \tau\right)\,.
\end{equation*}
Taking $\delta = n^{-1/2}$ yields
\begin{equation*}
	\E \sup_{g \in \cC^s(K;M)} \int   g(y) (\rd \nu_n(y) - \rd \nu(y)) \lesssim C_K Mn^{-1/2}(1 - \log(n^{-1/2} )) \lesssim C_K M n^{-1/2} \log n\,,
\end{equation*}
as claimed.
\end{proof}

\begin{lemma}\label{lem: genball_fix} Let $P$ and $Q$ be compactly supported, and let $(f_\eps,g_\eps)$ denote the optimal dual potentials corresponding to $S_\eps(P,Q)$. For any real number $s \geq 0$, the derivatives of $(f_{\eps},g_{\eps})$ up to order $s$ are bounded by $C_{s, d, K}(1 + \eps^{1-s})$ on any compact set $K$, where $C_{s,d,K} > 0$ is some constant independent of $\eps$.
\end{lemma}
\begin{proof}
It suffices to show the claim for $f_\eps$. Let $r$ be a positive integer, and let $\lambda \in [0,1]$. By \cite[Theorem 2]{GenChiBac18}, it holds that 
\begin{align*}
\|f_\eps\|_{\cC^r} = O(1 + \eps^{1-r})\,.
\end{align*}
For any $s \geq 0$, we can write $s = r + (1-\lambda)$ for some $\lambda \in (0,1)$ and $r \in \NN$. Consequently, any such $s$ can be written as $s = \lambda r + (1-\lambda)(r+1)$, from which we can now apply an interpolation inequality between the two integers \citep{lunardi2009interpolation}:
\begin{align*}
 \|f_\eps\|_{\cC^{\lambda r + (1-\lambda)(r+1)}} 
&\lesssim \|f_\eps\|_{\cC^r}^\lambda\|f_\eps\|_{\cC^{r+1}}^{1-\lambda} \\
&\lesssim (1 + \eps^{1-r})^{\lambda}(1 + \eps^{-r})^{1-\lambda} \\
&\leq 1 + \eps^{(1-r)\lambda - r(1-\lambda)} \\
&= 1 + \eps^{-r + \lambda} \\
&= 1 + \eps^{1-s}\,.
\end{align*} 
Thus, $\| f_\eps \|_{\cC^s} = O(1 + \eps^{1-s})$ for any $s \geq 0$, as desired.
\end{proof}

\begin{corollary}\label{cor: emp_proc_special}
	If $P$ and $Q$ are compactly supported, then
\begin{align*}
\E S_\eps(P, Q_n) - S_\eps(P, Q) \lesssim (1 + \eps^{1 - d/2}) \log(n) n^{-1/2}\,.
\end{align*}
\end{corollary}
\begin{proof}
	Let $(f_{\eps, n}, g_{\eps_n})$ be the optimal dual potentials for $P$ and $Q_n$.
	Following \cite[Proposition 2]{MenNil19}, observe that
\begin{align*}
S_\eps(\mu, \nu_n) -S_\eps(\mu, \nu) & = \int {f}_{(\eps,n)} \dd \mu + \int {g}_{(\eps,n)} \dd \nu_n -\sup_{f, g} \Big\{\int f \dd \mu + \int g \dd \nu \\
& \quad \quad - \eps \iint e^{(f(x) + g(y) - \frac 12 \|x - y\|^2)/\eps} \dd \mu(x) \dd \nu(y) + \eps\Big\} \\
& \leq \int {g}_{(\eps,n)}(y)(\rd \nu_n(y) - \rd \nu(y))\,,
\end{align*}
where the bound follows from choosing $({f}_{(\eps,n)},{g}_{(\eps,n)})$ in the supremum and using
\begin{equation*}
\int e^{({f}_{(\eps,n)}(x) + {g}_{(\eps,n)}(y) - \frac 12 \|x - y\|^2)/\eps} \dd \mu(x) = 1 \quad \forall y \in \RR^d
\end{equation*}
by the dual optimality condition \cref{dual_opt}.

We conclude by applying \Cref{lem: genball_fix}:
the derivatives of $g_{\eps, n}$ up to order $s$ are bounded by $C_{s, d, K}(1 + \eps^{1-s})$ on any compact set $K$ for any $s \geq 0$, so we may take $K$ to be a suitably large ball containing the support of $P$ and $Q$ and apply \Cref{empirical_process} with $s = d/2$.
\end{proof}

\section{Proof of Theorem \ref{thm: adaptive_estimation}}\label{sec: adaptivity}
We recall the notation from the main text. For convenience, we consider $\alpha \geq 1 + \iota$ for some $\iota>0$ sufficiently small, but fixed. Let $s := \alpha + 1$, which defines the regularity of the conjugate Brenier potential $\varphi_0^*$, thus $s \in [2 + \iota,4]$ for our problem considerations, since smoothness is capped at $\alpha=3$. Let $\cS$ be the following discrete subset
\begin{align*}
    \cS \defeq \{s_{\min} = s_1 < s_2 < \cdots < s_N = s_{\max}\}\,,
\end{align*}
where $s_{\min} = 2+\iota$, $s_N = 4$, with increments $s_j - s_{j-1} \asymp (\log n)^{-1}$, and set
\begin{align*}
    \eps_s = (n/\log n)^{-1/2(d+s)}\,, \quad \psi_n(s) = (\eps_s)^s = (n/\log n)^{-s/2(d+s)}\,.
\end{align*}
Let $\DD_n := \{(X_i,Y_i)\}_{i=1}^n$ denote our initial dataset with hold-out dataset $\DD_n'$. The latter gives rise to empirical measures $P_n'$ and $Q_n'$. Our choice of smoothness parameter is given by the following rule:
\begin{align}\label{eq: lepski_rule_app}
    \hat{s} \defeq \max\{ s \in \cS \ : \ \| \hat{T}_{\eps_{s}} - \hat{T}_{\eps_{s'}}\|_{L^2(P_n')}^2 \lesssim \psi_n(s')\,, \forall \ s' \leq s, s' \in \cS \}\,.
\end{align}

The proof closely follows an exposition of Lepski's method due to \cite{hutter2017notes}. 

For a given probability measure and its empirical counterpart from $n$ samples, written $\rho$ and $\rho_n$, we will frequently return to the empirical process over a given function class $\cM$, written
\begin{align*}
\| \rho - \rho_n\|_{\cM} \defeq \sup_{f \in \cM}\left| \int f \dd(\rho - \rho_n) \right|\,.
\end{align*}
We will consider the following function classes: $\cF_\eps$ will denote the class of entropic Kantorovich potentials for a regularization parameter $\eps$, and $\cJ_\eps$ be the function class from \cref{eq: J_eps_class}. $\cH_N$ will denote the random, $P_n$-measurable set of $N^2$ bounded functions of the form
\begin{align}
\| \hat{T}_{s_i}(x) - \hat{T}_{s_j}(x)\|^2_2\,,
\end{align}
for $i,j \in \{1,2,\ldots,N\}$, where we recall that $N$ is the cardinality of $\cS$.

Without loss of generality, we can assume $\varphi_0^* \in \cC^{s_i}$ for some $s_i \in \cS$. We define the event $\cE_j \defeq \{\hat{s} = s_j\}$ for all $j \in [N]$, and denote our estimator by $\hat{T}_{\hat{s}}$ (for clarity, we omit the explicit dependence on $\eps$). The ratio between the risk of $\hat{T}_{\hat{s}}$ and the oracle rate $\psi_n(s_i)$ can be written as 
\begin{align*}
    \E[\| \hat{T}_{\hat{s}} - T_0\|^2_{L^2(P)} \psi_n(s_i)^{-1}] %&= \sum_{j=1}^N \E[\| \hat{T}_{\hat{s}} - T_0\|^2_{L^2(P)} \psi_n(s_i)^{-1} \bm{1}(\cE_j)] \\
    &= \sum_{j=1}^{i-1} \E[\| \hat{T}_{s_j} - T_0\|^2_{L^2(P)} \psi_n(s_i)^{-1} \bm{1}(\cE_j)] \\
    &\quad + \sum_{j=i}^N \E[\| \hat{T}_{s_j} - T_0\|^2_{L^2(P)} \psi_n(s_i)^{-1} \bm{1}(\cE_j)]\,.
\end{align*}
Our goal is to show that the right-hand side is upper bounded by an absolute constant. We study the two terms above separately.

Let us first focus on the terms where $j \geq i$, i.e. our estimator of the smoothness of the optimal transport map is larger than the actual smoothness parameter. Inside the expectation, we can write via Young's inequality
\begin{align*}
\|\hat{T}_{s_j} - T_0\|^2_{L^2(P)} & \lesssim \|\hat{T}_{s_j} - \hat{T}_{s_i}\|^2_{L^2(P)} +  \| \hat{T}_{s_i} - T_0\|^2_{L^2(P)} \\
&= \|\hat{T}_{s_j} - \hat{T}_{s_i}\|^2_{L^2(P_n')} + \| \hat{T}_{s_i} - T_0\|^2_{L^2(P)} + \int \tilde{h}\dd(P - P_n') \\
&\leq  \|\hat{T}_{s_j} - \hat{T}_{s_i}\|^2_{L^2(P_n')} + \| \hat{T}_{s_i} - T_0\|^2_{L^2(P)} + \|P - P_n'\|_{\cH_N}\,,
\end{align*}
where $\tilde h = \|\hat{T}_{s_j} - \hat{T}_{s_i}\|^2_2$. We conclude by taking expectations. The first term on the right-hand side is bounded by $\psi_n(s_i)$: our estimator $\hat{s}=s_j$ under the event $\cE_j$, and our criterion for $\hat{s}$, namely \cref{eq: lepski_rule_app}, and $s_i \leq s_j$ by assumption.
For the second term: as $\phi_0^* \in \cC^{s_i}$, our main theorem (\cref{thm:main}) tells us that
$$\E\|\hat T_{s_i} - T_0\|_{L^2(P)}^2 \lesssim \psi_n(s_i)\,.$$ 
The third term, by Hoeffding's inequality and a union bound, satisfies
\begin{align*}
\E\|P_n' - P\|_{\cH_N} = \E[\E[\|P_n' - P\|_{\cH_N} \mid P_n]] \lesssim \log\log(n)/\sqrt{n}\,,
\end{align*}
where we used that $N \asymp \log n$. Note that the third term is in fact faster than any $\psi_n(s_i)$ for any choice of $s_i \in \cS$.
Altogether, this gives the following bound
\begin{align*}
        \sum_{j=i}^N \E[\| \hat{T}_{s_j} - T_0\|^2_{L^2(P)} \psi_n(s_i)^{-1} \bm{1}(\cE_j)] &\lesssim (c_0^2 +  \tilde{c}_0^2)\sum_{j=i}^N \PP(\cE_j) + \bar{c}_0^2 \leq C_0\,,
\end{align*}
for three different constants $c_0, \tilde{c}_0, \bar{c}_0 > 0$.

We now turn our attention to the case where $j < i$, which is more technical. Focusing on one term in the summand, we want to choose $t_j$ to appropriately balance
\begin{align*}
    \E[\|\hat{T}_{s_j} - T_0\|^2_{L^2(P)}\psi_n(s_i)^{-1}\bm{1}(\cE_j)] \leq t_j \PP(\cE_j) + \int_{t_j}^\infty \PP(\|\hat{T}_{s_j} - T_0\|^2_{L^2(P)}\psi_n(s_i)^{-1} \geq t) \dd t\,.
\end{align*}

By definition of the estimator, we can upper bound $\PP(\cE_j)$ by two events, leading to 
\begin{align}\label{eq: proba_bound_adaptive}
    \PP(\cE_j) \leq \sum_{l=1}^{i-1} \left( \PP(\|\hat{T}_{s_i} - T_0\|^2_{L^2(P_n')}\psi_n(s_l)^{-1} > c_0^2/4) + \PP(\|\hat{T}_{s_l} - T_0\|^2_{L^2(P_n')}\psi_n(s_l)^{-1} > c_0^2/4) \right)\,.
\end{align}
Indeed, since $s_j < s_i$ and since we are on the set $\cE_j$, there must exist an $s_j < s' < s_i$ such that 
\begin{align*}
\| \hat{T}_{s_i} - \hat{T}_{s'}\|^2_{L^2(P_n')} \psi_n(s') > c_0.
\end{align*}
By Young's inequality, we can break this up into two possible events, whereby summing over all possible $s'$ gives the above bound (we replace $s'$ by $s_l$). Finally, we note that we also have the inequality
\begin{align*}
\PP(\|\hat{T}_{s_i} - T_0\|^2_{L^2(P_n')}\psi_n(s_l)^{-1} > c_0^2/4) \leq \PP(\|\hat{T}_{s_i} - T_0\|^2_{L^2(P_n')}\psi_n(s_i)^{-1} > c_0^2/4)\,,
\end{align*}
since $\psi_n(\cdot)$ is decreasing. It remains to bound these two tail probabilities across all $l < i$, where note the norm is measured in $L^2(P_n')$. To continue, we require the following lemma.
\begin{proposition}\label{prop: tail_bound}
There exist absolute constants $c, C > 0$ such that for $t \geq c$,
\begin{align*}
    \PP\left(\|\hat{T}_s - T_0\|^2_{L^2(P)} \psi_n(s)^{-1} \geq ct \right) \leq \exp\left(-\frac{t^2 \log(n)}{C}\right).
\end{align*}
\end{proposition}
\begin{proof}
For any choice of $s \in (2,4]$, it holds that 
\begin{align*}
\|\hat{T}_s - T_0\|^2_{L^2(P)}\lesssim \eps^{s/2} + \|Q_n - Q\|_{\cJ_\eps} + \eps^{-1}\|P_n - P\|_{\cF_\eps}\,,
\end{align*}
which stems from the calculations that appear between \cref{thm:one_sample} and \cref{thm:two_sample}. Both $\|Q_n - Q\|_{\cJ_\eps}$ and $\|P_n - P\|_{\cF_\eps}$ are subGaussian random variables via McDiarmid's inequality: for two constants $a,b>0$, it holds that for $t$ large enough 
\begin{align*}
    &\PP(\|Q_n - Q\|_{\cJ_\eps} \geq (1+t)(\eps^{-d}n^{-1})^{1/2}) \leq e^{-at^2/2}\,,\\
    &\PP(\eps^{-1}\|P_n - P\|_{\cF_\eps} \geq (\eps^{-1}n^{-1})^{1/2}t + \eps^{-d/2}n^{-1/2} )\leq e^{-bt^2/2}\,.
\end{align*}
Consequently, we can merge these via a union bound; taking the worst case constant, we have that for $t \geq c\eps^{-d/2}n^{-1/2}$, for $c > 0$ sufficiently large, it holds that
\begin{align*}
    \PP(\| \hat{T}_s - T_0\|^2_{L^2(P)} \gtrsim \eps^{s/2} + \eps^{-d/2}n^{-1/2}t) \leq e^{-ct^2/2}\,.
\end{align*}
Dividing through by $\psi_n(s) \defeq (n/\log(n))^{-\frac{s}{2(d+s)}}$ completes the proof.
\end{proof}

We can also obtain tail bounds under $L^2(P_n')$ at virtually no cost. Indeed, for any $s \in \cS$, 
\begin{align*}
\| \hat{T}_{s} - T_0\|^2_{L^2(P_n')} \lesssim \| \hat{T}_{s} - T_0\|^2_{L^2(P)} +\| P_n' - P\|_{\cH_N}\,,
\end{align*}
where the last term has expectation bounded above by $\log\log(n)n^{-1/2}$ up to a constant factor (indeed, since $T_0 = T_{s_i}$, this is perfectly fine at the cost of adding one more function to the set). By employing a further union bound, we can state \cref{prop: tail_bound} as
\begin{align}\label{eq: tail_bound_2}
    \PP\left(\|\hat{T}_s - T_0\|^2_{L^2(P_n')} \psi_n(s)^{-1} \geq ct \right) \leq 2\exp\left(-\frac{t^2 \log(n)}{C}\right),
\end{align}
for any $s \in \cS$, where the constants that appear are slightly different. Indeed, since $\log \log(n)/\sqrt n \ll \psi_n(s)$, nothing is lost by incorporating this additional term.

Returning to \cref{eq: proba_bound_adaptive}, we can take $c_0$ sufficiently large in both terms, we can employ \cref{eq: tail_bound_2} for all the terms in the summand, which results in 
\begin{align*}
\PP(\cE_j) \leq n^{-c_0^2/(8C)}\,.
\end{align*}

For the integrated tail, we use a similar argument, appealing to \cref{prop: tail_bound} directly. Indeed, for $t \geq C \psi_n(s_j)/\psi_n(s_i)$, the following bound holds:
\begin{align}
    \PP\left(\|\hat{T}_{s_j} - T_0\|^2_{L^2(P)} \psi_n(s_i)^{-1} \geq t \right) \leq \exp\left(-\frac{t^2 \log(n)}{C} \frac{\psi_n(s_i)}{\psi_n(s_j)} \right).
\end{align}
Choosing $t_j =c_1\sqrt{\psi_n(s_j)/\psi_n(s_i)}$, the tail can be upper bounded as
\begin{align*}
\int_{t_j}^\infty \exp\left(-\frac{t^2 \log(n)}{C} \frac{\psi_n(s_i)}{\psi_n(s_j) }\right) \dd t &\leq \left( \frac{\psi_n(s_j) C}{\psi_n(s_i) \log(n)} \right) \sqrt{ \frac{\psi_n(s_i)}{\psi_n(s_j)c_1^2}}\exp\left(-\frac{c_1\log(n)}{C}\right) \\
&= \sqrt{ \frac{\psi_n(s_j)}{\psi_n(s_i)} }\frac{C}{c_1\log(n)} \exp\left(-\frac{c_1\log(n)}{C}\right)\,.
\end{align*}
Merging everything together, we obtain rather crudely that
\begin{align*}
\sum_{j=1}^{i-1}\E[\|\hat{T}_{s_j} - T_0\|^2_{L^2(P)}\psi_n(s_i)^{-1}\bm{1}(\cE_j)] &\leq \sum_{j=1}^{i-1}\left(t_j n^{-c_0^2/(8C)} + \sqrt{ \frac{\psi_n(s_j)}{\psi_n(s_i)} }\frac{C}{c_1\log(n)}\exp\left(-\frac{c_1\log(n)}{C}\right)\right) \\
&\leq \sum_{j=1}^{i-1}\frac{1}{\log n}\\
&\asymp 1\,,
\end{align*}
since there exist $N \asymp \log(n)$ terms. This completes the proof.

\end{appendix}

\bibliographystyle{abbrvnat_weed}
\bibliography{arammain}

\end{document}